\documentclass{article}

\usepackage{amsmath,amsthm,amsfonts}
\usepackage{amssymb}
\usepackage{mathrsfs}
\newtheorem{theorem}{Theorem}[section]
\newtheorem{lemma}{Lemma}[section]
\newtheorem{corollary}{Corollary}[section]
\newtheorem{definition}{Definition}[section]
\numberwithin{equation}{section}

\begin{document}

\title{\bf Some new results on generalized diagonally dominant matrices and matrix eigenvalue inclusion regions}

\author{Yongzhong Song\\\\
\small{\it Jiangsu Key Laboratory for NSLSCS,
School of Mathematical Sciences,} \\
\small{\it  Nanjing Normal University,
Nanjing 210023, People's Republic of China}\\
\small{\it Email:yzsong@njnu.edu.cn}
}
\date{}
\maketitle

\begin{abstract}
\vskip.5cm

In matrix theory and numerical analysis there are two very famous and important results.
One is Ger\v{s}gorin circle theorem, the other is strictly diagonally dominant theorem.
They have important application and research value, and have been widely used and studied.
In this paper, we investigate generalized diagonally dominant matrices and matrix eigenvalue
inclusion regions. A class of $G$-function pairs is proposed, which extends the concept of
$G$-functions. Thirteen kind of $G$-function pairs are established. Their properties and
characteristics are studied. By using these special $G$-function pairs, we construct a large
number of sufficient and necessary conditions for strictly diagonally dominant matrices and
matrix eigenvalue inclusion regions. These conditions and regions are composed of different
combinations of $G$-function pairs, deleted absolute row sums and column sums of matrices.
The results extend, include and are better than some classical results.


\vskip.5cm
{\it AMS classification (2010):}
15A18,  15A42,  65F15,  26D07

\end{abstract}

\newpage
\pagestyle{headings}
\markboth{}{}
\hskip.1cm
 \vskip .5cm\tableofcontents

 \vskip .5cm\noindent
 {\bf References  \hfill 42}

\newpage\markboth{}{}
\section{Introduction}

In matrix theory and numerical analysis there are two very famous and important results.
One is Ger\v{s}gorin circle theorem, the other is strictly diagonally dominant theorem.
They have important application and research value, and have been widely used and studied.

Let
\begin{eqnarray*}
A = \left[\begin{array}{ccccc}
  a_{1,1} & a_{1,2} & \cdots & a_{1,n}\\
  a_{2,1} & a_{2,2} & \cdots & a_{2,n}\\
  \vdots & \vdots &   &  \vdots \\
  a_{n,1}& a_{n,2} &  \cdots & a_{n,n}
\end{array} \right] \in \mathbb{C}^{n\times n}.
 \end{eqnarray*}

Ger\v{s}gorin [20] proposed the following famous eigenvalue
inclusion theorem.

\medskip

{\bf Ger\v{s}gorin Circle Theorem:}\medskip

{\it For any eigenvalue $\lambda\in\sigma(A)$, there exists $i\in\{1,\cdots,n\}$
such that
\begin{eqnarray*}
|\lambda - a_{i,i}| \le \sum_{1\le j\le n, j\not= i}|a_{i,j}|.
 \end{eqnarray*}
Consequently,
\begin{eqnarray*}
\sigma(A) \subseteq \bigcup_{i=1}^n\left\{z: |z - a_{i,i}| \le
\sum_{1\le j\le n, j\not= i}|a_{i,j}|\right\}.
 \end{eqnarray*} }

This theorem is simple in form and easy to determine numerically,
so it is widely used.

Since the Ger\v{s}gorin circle theorem was published,
a lot of researches and extensions have been made by many
mathematicians. People use different combinations (sum, product, convex combination, etc.)
of absolute row/column sums (or their parts) to obtain new eigenvalue inclusion regions.
Some early results are collected [13].
A large number of relevant results are reviewed by Varga in the comprehensive monograph [52].

In order to generalize the Ger\v{s}gorin circle theorem, $G$-functions are defined
in two different ways [7, 25],
where $G$ comes from the first letter of Ger\v{s}gorin [24].
This concept is first proposed by Nowosad (1965) and Hoffman (1969).
The original name is $G$-generating family.
After that, it is studied by many mathematicians [7, 8, 21, 22, 24, 25, 32, 33].

The Ger\v{s}gorin circle theorem can be proved in different ways, and one of the simple
methods is based on the nonsingularity of matrices [31, III-2.2.1].
One of the two definitions of $G$-functions is also expressed by the nonsingularity of matrices [7, 25].
In mathematics, nonsingularity is a very important property of matrices.
The following theorem provides a sufficient condition for a matrix to be nonsingular.

 \bigskip
{\bf Strictly Diagonally Dominant Theorem:}\medskip

{\it If $A$ is strictly diagonally dominant, i.e.,
\begin{eqnarray*}
|a_{i,i}| > \sum_{1\le j\le n, j\not= i}|a_{i,j}|, \quad i =1,
\cdots, n,
 \end{eqnarray*}
then $A$ is nonsingular.}

\bigskip
Strictly diagonally dominant is initially called strongly diagonally dominant [19, 30].

The strictly diagonally dominant theorem is an old and recurring result in matrix theory.
It can be traced back to at least L\'{e}vy (1881), Desplanques (1887), Minkowski (1900) and Hadamard (1903).
So it is also called as the Hadamard's Theorem [35, 38] and
the L\'{e}vy-Desplanques Theorem [6, 27, 31].
In [23], it is named as Desplanques-L\'{e}vy-Hadamard-Ger\v{s}gorin sufficient condition.
In [20, Satz I], the strictly diagonally dominant theorem is stated and proved.
However, there are some mistakes in content and proof. The corrected result and
a very simple proof is given in [46], where it is also extended to the irreducibly diagonally dominant matrices.

We notice an interesting phenomenon that both the Ger\v{s}gorin circle theorem
and the strictly diagonally dominant theorem are closely related to two
groups of numbers $\{a_{i,i}: i = 1,\cdots,n\}$ and
$\{\sum_{1\le j\le n, j\not= i}|a_{i,j}|: i = 1,\cdots,n\}$.
Therefore, the equivalence between the Ger\v{s}gorin circle
theorem and the strictly diagonally dominant theorem is shown
[6, 50, 52]. In other words, the Ger\v{s}gorin circle
theorem can be derived from the strictly diagonally dominant theorem and vice versa.
This provides a very meaningful way, that is, through the study of strictly diagonally
dominant matrix to derive matrix eigenvalue inclusion regions.

The strictly diagonally dominant theorem can be extended to generalized diagonally dominant matrix.
The name of generalized diagonally dominant matrix comes from [28, Definition 2].
While, due to [34], the generalized diagonally dominant matrices and $H$-matrices are the same
(cf. also [19, (1.2)], [30, Lemma 1.4].

$H$-matrix was first introduced and studied by Ostrowski [34, 35, 36, 37],
under the names of ``$H$-Determinante" and ``$H$-Matrix", respectively.
$H$-matrices play a very important role in numerical analysis,
optimization theory and other applied sciences.
It has not only important theoretical research value, but also extensive application value.
For classical iterative methods of linear systems,
such as Jacobi, Gauss-Seidel, SOR and AOR methods, etc.,
$H$-matrices are widely used to construct the sufficient conditions for convergence
[3, 28,30, 37, 42, 44, 45, 49].
For a linear complementarity problem (LCP), when
the coefficient matrix is a real $H$-matrix with positive diagonal elements,
then it has a unique solution and we can construct a class of convergent modified AOR methods
[54].
Furthermore, $H$-matrices are closely related to $M$-matrices.
$M$-matrix was first introduced and studied by Ostrowski [34, 35, 36, 37],
under the names of ``$M$-Determinante" and ``$M$-Matrix".
The properties of $M$-matrices, $H$-matrices and related materials are summarized
[3, 17, 26, 39, 40, 41, 49].

In this paper, we investigate generalized diagonally dominant matrices and matrix eigenvalue inclusion regions.
A class of $G$-function pairs is proposed, which extends the concept of $G$-functions.
For general $G$-function pairs, we prove their relations with strictly
diagonally dominant matrices and the matrix eigenvalue inclusion regions, respectively.
Thirteen kind of $G$-function pairs are established.
Their properties and characteristics are studied, and their relations with
$G$-functions are discussed. By using
these special $G$-function pairs, we construct a large number of sufficient and necessary
conditions for strictly diagonally dominant matrices and matrix eigenvalue inclusion regions.
These conditions and regions are composed of different combinations of $G$-function pairs,
deleted absolute row sums and column sums of matrices. Our results extend, include and are
better than some classical results.

This paper is organized as follows. In Section 2 we give
some concepts and lemmas to be used in the following.
In Section 3, we propose the definition of $G$-function pairs and prove
the relations with strictly diagonally dominant matrices.
We establish thirteen kind of $G$-function pairs and discuss their properties, characteristics and relations with
$G$-functions. In Section 4, a large of necessary and
sufficient conditions for strictly diagonally dominant matrices are constructed.
In Section 5, we prove the relations between the general $G$-function pairs and the matrix eigenvalue
inclusion regions. Many inclusion regions are established.
In Section 6, some reviews and prospects are given.

\newpage\markboth{}{}
\section{Some concepts and lemmas}\label{se2}

For any positive integer $n$, denote ${\cal N}=\{1, 2, \cdots, n\}$.
Without loss of generality, we assume that $n \ge 2$.

Let $\mathbb{C}^{n\times n}$ and $\mathbb{R}^{n\times n}$ denote the collection of all $n \times n$ matrices
with complex and real entries, respectively.
Let $\mathbb{C}^{n}$ and $\mathbb{R}^{n}$ denote the collection of all column vectors
with complex and real entries, respectively.
For $A=[a_{i,j}]\in \mathbb{C}^{n\times n}$,
$A^T$ is the transposition of $A$, $|A|:=[|a_{i,j}|]$,
$D(A)=diag[a_{1,1},a_{2,2},\cdots,a_{n,n}]$.
Let $D^P_n$ be the set of all diagonal matrices of order $n$ with
positive diagonal entries. Denote $\mathbb{R}_+^n$ as the set of column
vectors with nonnegative entries and $e=[1,1,\ldots,1]^T\in \mathbb{R}_+^n$.
The spectrum of $A$, denoted by $\sigma(A)$, is the collection of
eigenvalues of $A$. Spectral radius of $A$ is denoted by $\rho(A)$.

For $A=[a_{i,j}], B=[b_{i,j}]\in \mathbb{R}^{n\times
n}$, notation $A \ge (>) B$ means $a_{i,j} \ge (>) b_{i,j}$ for all
$i,j\in {\cal N}$. We call $A$ nonnegative if $A \ge 0$.
Similarly, for $x=[x_1,x_2,\cdots,x_n]^T, y=[y_1,y_2,\cdots,y_n]^T\in \mathbb{R}^n$, notation $x \ge (>)y$ means
$x_i \ge (>)y_i$ for all $i\in {\cal N}$.

We call
\begin{eqnarray*}
r_i(A):=\sum_{j\in {\cal N}\setminus \{i\}}|a_{i,j}|,\quad
c_i(A):=\sum_{j\in {\cal N}\setminus \{i\}}|a_{j,i}|,
\end{eqnarray*}
the $i$-th deleted absolute row sum and column sum of $A$,
respectively. The ``row sum'' function $r$ and ``column sum'' function $c$
are defined as $r(A)=[r_1(A),r_2(A),\cdots,r_n(A)]^T$
and $c(A)=[c_1(A),c_2(A),\cdots,c_n(A)]^T$, respectively.
Clearly, $r(A),c(A) \in \mathbb{R}_+^n$ and $c(A)=r(A^T)$.

\begin{definition}
A matrix $A=[a_{i,j}]\in \mathbb{C}^{n\times n}$ is strictly
diagonally dominant (by rows), denoted by $A\in SDD$, if
\begin{eqnarray*}
|a_{i,i}|> r_i(A), \; \hbox{for all} \; i\in {\cal N},
\end{eqnarray*}
i.e.,
\begin{eqnarray*}
|D(A)|e > r(A).
\end{eqnarray*}
\end{definition}

The definition of generalized diagonally dominant matrix is proposed [28, Definition 2].
We give an equivalent definition as follows.

\begin{definition}
A matrix $A=[a_{i,j}]\in \mathbb{C}^{n\times n}$ is generalized
diagonally dominant (by rows), denoted by $A\in GDD$, if there
exists a diagonal matrix $X=diag[x_1,x_2,\cdots,x_n]\in D^P_n$ such that
$X^{-1}AX\in SDD$, i.e.,
\begin{eqnarray*}
|a_{i,i}|> r_i^X(A):=r_i(X^{-1}AX)=\frac{1}{x_i}\sum_{j\in
{\cal N}\setminus \{i\}}|a_{i,j}|x_j, \; \hbox{for all} \; i\in {\cal N}.
\end{eqnarray*}
\end{definition}

Here $r_i^X(A)$ is called the $i$-th weighted deleted absolute row sum of $A$ and
$r^X = (r_1^X, \cdots, r_n^X)^T$ is called the weighted row sum function.
Similarly, we define the $i$-th weighted deleted absolute column sum $c_i^X(A)$ of $A$ as
$c_i(X^{-1}AX)$ and the weighted column sum function as $c^X = (c_1^X, \cdots, c_n^X)^T$.

By strictly diagonally dominant theorem, it is clearly that if $A\in GDD$,
then it is nonsingular.

\begin{definition}
A matrix $A=[a_{i,j}]\in \mathbb{R}^{n\times n}$ is called a $Z$-matrix, denoted
by $A\in {\cal Z}$, if $a_{i,j} \le 0$  for all $i,j \in {\cal N}$, $i \not= j$.
\end{definition}

Obviously, for any $A\in {\cal Z}$, there exist a nonnegative matrix $B$
and a nonnegative number $s$ such that $A=sI-B$.

\begin{definition}
Given $A=[a_{i,j}]\in {\cal Z}$, express $A$ as $A=sI-B$, where $s \ge 0$
and $B \ge 0$. Then $A$ is an $M$-matrix, if
$\rho(B)<s$.
\end{definition}

We call $\mathcal{M}(A) \in \mathbb{R}^{n\times n}$ defined by
\begin{eqnarray*}
\mathcal{M}(A):=\left[\begin{array}{cccc}
 |a_{1,1}|&-|a_{1,2}|& \cdots &-|a_{1,n}|\\
-|a_{2,1}|&|a_{2,2}|& \cdots &-|a_{2,n}|\\
\vdots&\vdots& \ddots &\vdots\\
-|a_{n,1}|&-|a_{n,2}|& \cdots &|a_{n,n}|
\end{array}\right]
\end{eqnarray*}
as the comparison matrix of $A$.

\begin{definition}[[50, Definition 3.26]]
A matrix $A\in \mathbb{C}^{n\times n}$ is an $H$-matrix if $\mathcal{M}(A)$ is an $M$-matrix.
\end{definition}

It is well-known that $A$ is an $H$-matrix if and only if $A\in GDD$ [19, 30, 34]. Hence, an $H$-matrix
is nonsingular. Furthermore, a definition of $H$-matrix is given by the generalized diagonally dominant matrix
[1, Definition 1.1]. In particular, we have
 \begin{eqnarray*}
A \in GDD \;\; \hbox{iff} \;\; A^T \in GDD.
\end{eqnarray*}

\begin{definition}
A matrix $A=[a_{i,j}]\in \mathbb{C}^{n\times n}$ is reducible if
there exist a permutation matrix $P$ and a positive integer $\kappa$
with $1 \le \kappa < n$ such that
\begin{eqnarray*}
P^TAP=\left[\begin{array}{cc}
 A_{1,1}& A_{1,2}\\
 0 & A_{2,2}
 \end{array}\right],
\end{eqnarray*}
where $A_{1,1}\in \mathbb{C}^{\kappa\times \kappa}$, $A_{2,2}\in
\mathbb{C}^{(n-\kappa)\times (n-\kappa)}$. Otherwise,
$A$ is said to be irreducible.
\end{definition}

Clearly, if $A$ is reducible, then there exists a permutation matrix
$P$ such that
\begin{eqnarray}\label{eqn2-1}
P^TAP=\left[\begin{array}{cccc}
 A_{1,1}& A_{1,2} & \cdots&A_{1,l}\\
 0& A_{2,2}&\cdots & A_{2,l}\\
\vdots&\vdots& &\vdots\\
0&0&\cdots& A_{l,l}
\end{array}\right],
\end{eqnarray}
where $A_{k,k} \in \mathbb{C}^{(n_k-n_{k-1})\times (n_k-n_{k-1})}$
with $k\in \{1,\cdots,l\}$, $n_0=0$ and $n_l=n$ is an irreducible matrix
or zero matrix of order one. The form (2.1) is called
the Frobenius normal form of $A$.

We denote
\begin{eqnarray*}
&& \tilde{r}_{n_{k-1}+\eta}(A):=r_\eta(A_{k,k}), \; \tilde{c}_{n_{k-1}+\eta}(A):=c_\eta(A_{k,k}),\\
&& k =1,\cdots, l, \; \eta = 1, \cdots, n_k-n_{k-1}
\end{eqnarray*}
and
\begin{eqnarray*}
\tilde{r} :=[\tilde{r}_1,\cdots,\tilde{r}_n]^T, \;\;
\tilde{c}:=[\tilde{c}_1,\cdots,\tilde{c}_n]^T.
\end{eqnarray*}

We set $\tilde{r}^X(A):=\tilde{r}(X^{-1}AX)$ and
$\tilde{c}^X(A):=\tilde{c}(X^{-1}AX)$ for $X \in D^P_n$.
Obviously, $r(A) \ge \tilde{r}(A)$, $c(A) \ge \tilde{c}(A)$, $r^X(A) \ge \tilde{r}^X(A)$,
$c^X(A) \ge \tilde{c}^X(A)$.
Furthermore, $\tilde{r}(A)=r(A)$, $\tilde{c}(A)=c(A)$, $\tilde{r}^X(A)=r^X(A)$,
$\tilde{c}^X(A)=c^X(A)$, whenever $A$ is irreducible.

The following lemma is easy to prove.

\begin{lemma} \label{lem2-1}
Let $A\in \mathbb{C}^{n\times n}$ have the Frobenius normal form given by (2.1).
Then

\begin{itemize}
\item[\rm(i)]
$A$ is nonsingular if and only if $A_{k,k}$ is nonsingular for $k=1,\cdots,l$;

\item[\rm(ii)]
$A\in GDD$ iff $A_{k,k}\in GDD$
or $A_{k,k}$ is a non-zero constant for $k=1,\cdots,l$.
\end{itemize}
\end{lemma}

\begin{definition}[[52, Definition 5.1]]
$\mathcal{F}_n$ is defined as the collection of all functions $f=[f_1,
f_2,\cdots,f_n]^T$ such that

\begin{itemize}
\item[\rm(i)] $f:\mathbb{C}^{n\times n} \rightarrow \mathbb{R}^n_+$, i.e., for any $A\in \mathbb{C}^{n\times n}$,
$0 \le f_k(A)<+\infty$, $k\in {\cal N}$;

\item[\rm(ii)] for each $k\in {\cal N}$, $f_k(A)$ depends only on the moduli of the
off-diagonal entries of $A=[a_{i,j}]\in \mathbb{C}^{n\times n}$.
\end{itemize}
\end{definition}

\begin{definition}[[52, Definition 5.2]] \label{den2-8}
Let $f=[f_1,f_2,\cdots,f_n]^T\in\mathcal{F}_n$. Then $f$ is called a
$G$-function if for any $A=[a_{i,j}]\in \mathbb{C}^{n\times n}$ the
relation
\begin{eqnarray*}
|D(A)|e > f(A)
\end{eqnarray*}
implies that $A$ is nonsingular.
The set of all $G$-functions in $\mathcal{F}_n$
is denoted by $\mathcal{G}_n$.
\end{definition}

The following results are shown [7].

\begin{lemma}  \label{lem2-2}
Supposed that $r$, $c$, $\tilde{r}$ and $\tilde{c}$ are defined above. Then

\begin{itemize}
\item[\rm(i)] $r, c, \tilde{r}, \tilde{c} \in \mathcal{G}_n$;

\item[\rm(ii)] $r^X, c^X, \tilde{r}^X, \tilde{c}^X \in \mathcal{G}_n$ for any $X\in D^P_n$.
\end{itemize}
\end{lemma}

\begin{lemma}[[13] and [52, p.143]]  \label{lem2-3}
If $f = [f_1, f_2,\cdots, f_n]^T \in \mathcal{G}_n$,
then for any $A \in \mathbb{C}^{n\times n}$,
there exists $X\in D^P_n$ (depending on $A$) such that
\begin{eqnarray*}
f_k(A) \ge \tilde{r}_k^X(A), \; \mbox{for all} \; k\in {\cal N}.
\end{eqnarray*}
\end{lemma}

Accordingly, with the same proof in [52, p.143], the following lemma can be proved.

\begin{lemma}  \label{lem2-4}
If $f = [f_1,f_2, \cdots, f_n]^T \in \mathcal{G}_n$,
then for any $A \in \mathbb{C}^{n\times n}$,
there exists $Y \in D^P_n$ (depending on $A$) such that
\begin{eqnarray*}
f_k(A) \ge \tilde{c}_k^Y(A), \; \mbox{for all} \; k\in {\cal N}.
\end{eqnarray*}
\end{lemma}

The following generalized arithmetic-geometric mean inequality will be applied later.

\begin{lemma}[[2, (4.30)]]  \label{lem2-5}
If $x,y \ge 0$ and $0 \le\alpha \le 1$,
then
\begin{eqnarray*}
 \alpha x + (1-\alpha)y \ge x^{\alpha}y^{1-\alpha}.
 \end{eqnarray*}
\end{lemma}

Suppose that $g = [g_1,g_2, \cdots, g_n]^T, h = [h_1,h_2,
\cdots, h_n]^T \in \mathcal{F}_n$ and $0 \le\alpha \le 1$.
Their $\alpha$-convolution is defined as $f=g^{\alpha}h^{1-\alpha}$,
where $f = [f_1,f_2, \cdots, f_n]^T$ with $f_k(A)=g_k^{\alpha}(A)h_k^{1-\alpha}(A)$
for $k\in {\cal N}$. And their $\alpha$-weighted sum is defined as $\tilde{f}=\alpha g + (1-\alpha)h$,
where $\tilde{f} = [\tilde{f}_1,\tilde{f}_2, \cdots, \tilde{f}_n]^T$ with $\tilde{f}_k(A)=\alpha g_k(A) + (1-\alpha)h_k(A)$
for $k\in {\cal N}$. Clearly, $f,\tilde{f}\in \mathcal{F}_n$.

\begin{lemma}[[7, Theorem 1]]  \label{lem2-6}
If $g,h, \in \mathcal{G}_n$ and if $0 \le\alpha \le 1$,
then $f=g^{\alpha}h^{1-\alpha}\in \mathcal{G}_n$.
\end{lemma}

By Lemmas 2.5 and 2.6, the following lemma can be proved directly.

\begin{lemma}[[7, [p.100]] \label{lem2-7}
If $g,h, \in \mathcal{G}_n$ and if $0 \le\alpha \le 1$,
then $\tilde{f}=\alpha g + (1-\alpha)h \in \mathcal{G}_n$.
\end{lemma}

\newpage\markboth{}{}
\section{$G$-function pairs} \label{se3}

In this section, referring to [47], we propose a class of $G$-function pairs and discuss their properties.

We first give the definition and some basic results.

\begin{definition} \label{den3-1}
A function $F: \mathbb{R}_+^n\times \mathbb{R}_+^n\rightarrow \mathbb{R}_+^m$
for $m \ge 1$ is monotonic if $F(x,y) \ge F(u,v)$ for any $x, y, u, v\in
\mathbb{R}_+^n$ satisfying $x \ge u$ and $y \ge v$.
\end{definition}

\begin{definition}\label{den3-2}
A function pair $(g,h)\in \mathcal{F}_n\times \mathcal{F}_n$ is called an
$G$-function pair induced by a monotonic function
$F: \mathbb{R}_+^n\times \mathbb{R}_+^n\rightarrow \mathbb{R}_+^m$ for $m \ge 1$ if,
for any $A = [a_{i,j}] \in {\mathbb{C}^{n\times n}}$, the relations
\begin{eqnarray}\label{eqnn3-1}
F(|D(A)|e,|D(A)|e)>F(g(A),h(A))
\end{eqnarray}
implies that $A$ is nonsingular. The set of the $G$-function pair induced by $F$ is
denoted by $\mathcal{G}^F$.
\end{definition}

From the definition of $G$-function pairs, if $A\in {\mathbb{C}^{n\times n}}$ satisfies the inequality
(3.1) for some $(g,h)\in \mathcal{G}^F$, then
it is nonsingular. Furthermore, we prove the following strong results [47, Theorem 1.3.2].

\begin{theorem}\label{thm3.1}
If $A\in {\mathbb{C}^{n\times n}}$ satisfies the inequality
(3.1) for some $(g,h)\in \mathcal{G}^F$, then
$A\in GDD$.
\end{theorem}

\begin{proof}
Let $\mathcal{M}(A) = sI - B$ with $s \ge 0$ and $B \ge
0$. For $t \ge 0$, let $A_t = tI + \mathcal{M}(A)$.

Since $D(A_t)e = D(\mathcal{M}(A))e + te = |D(A)|e + te \ge
|D(A)|e$, and the off-diagonal entries of $A_t$ and $A$ are same,
thus $A_t$ satisfies the inequality (3.1) from the
monotony of $F$. Therefore $A_t$ is nonsingular for any
$t \ge 0$.

Suppose $\rho(B) \ge s$. Then there exists $t_0 \ge 0$ such that
$\rho(B) = s + t_0$. This shows that $A_{t_0} = (s+t_0)I - B$ is
singular and contradictory. Thus we have proved $\rho(B) < s$, so that
$\mathcal{M}(A)$ is an $M$-matrix and $A\in GDD$.
\end{proof}

The following theorem provides a criterion of choosing
$F$ to ensure meaningfulness and generality of the
concept of the $G$-function pairs [47, Theorem 1.3.4].

\begin{theorem}\label{thm3.2}
$\mathcal{G}_n\times \mathcal{G}_n \subseteq
\mathcal{G}^F$ iff $(\tilde{r}^X, \tilde{r}^Y)\in
\mathcal{G}^F$ for arbitrary $X,Y\in D^P_n$.
\end{theorem}

\begin{proof}
Assume that $\mathcal{G}_n\times
\mathcal{G}_n \subseteq \mathcal{G}^F$. By Lemma 2.2,
$\tilde{r}^X \in \mathcal{G}_n$ for any $X \in D^P_n$. Hence we
have $(\tilde{r}^X, \tilde{r}^Y)\in \mathcal{G}^F$ for
any $X,Y\in D^P_n$.

Conversely, suppose that $(\tilde{r}^X, \tilde{r}^Y)\in \mathcal{G}^F$ for any $X,Y\in D^P_n$.
For any given $(g,h)\in \mathcal{G}_n\times \mathcal{G}_n$, assume that $A\in \mathbb{C}^{n\times n}$ satisfies
\begin{eqnarray*}
F(|D(A)|e,|D(A)|e) > F(g(A),h(A)).
\end{eqnarray*}

From Lemma 2.3, there exist $X,Y \in D^P_n$ (depending
on $A$) such that $g(A) \ge \tilde{r}^X(A)$ and $h(A) \ge
\tilde{r}^Y(A)$, so that
\begin{eqnarray*}
F(|D(A)|e,|D(A)|e) > F(\tilde{r}^X(A),\tilde{r}^Y(A))
\end{eqnarray*}
from the monotony of $F$.
Since $(\tilde{r}^X, \tilde{r}^Y)\in \mathcal{G}^F$,
it follows that $A$ is nonsingular. Thus $(g,h)\in
\mathcal{G}^F$, so that $\mathcal{G}_n\times
\mathcal{G}_n \subseteq \mathcal{G}^F$.
\end{proof}

Similarly, using Lemma 2.4, we can prove the following theorem.

\begin{theorem}\label{thm3.3}
The following conclusions are valid:

\begin{itemize}
\item[\rm(i)]
$\mathcal{G}_n \times \mathcal{G}_n \subseteq \mathcal{G}^F$ iff
$(\tilde{c}^X,\tilde{c}^Y)\in \mathcal{G}^F$ for arbitrary $X,Y\in D^P_n$;

\item[\rm(ii)]
$\mathcal{G}_n \times \mathcal{G}_n \subseteq \mathcal{G}^F$ iff
$(\tilde{r}^X,\tilde{c}^Y)\in \mathcal{G}^F$ for arbitrary $X,Y\in D^P_n$;

\item[\rm(iii)]
$\mathcal{G}_n \times \mathcal{G}_n \subseteq \mathcal{G}^F$ iff
$(\tilde{c}^X,\tilde{r}^Y)\in \mathcal{G}^F$ for arbitrary $X,Y\in D^P_n$.
\end{itemize}
\end{theorem}

Now, we propose some kinds of $G$-function pairs.
Their properties and characteristics are studied, and their relations with
$G$-functions are discussed.

\begin{definition}   \label{den3-3}
For $x=(x_1,x_2,\cdots,x_n)^T, y=(y_1,y_2,\cdots, y_n)^T \in \mathbb{R}_+^n$
and $0 \le \alpha,\beta \le 1$,
we define monotonic functions $F_{\mu,\alpha}, F_{\nu,\alpha,\beta}: \mathbb{R}_+^n\times \mathbb{R}_+^n\rightarrow \mathbb{R}_+^{n(n-1)}$, $\mu=1,\cdots,7$, $\nu = 8, \cdots, 13$,
as follows:

\begin{itemize}
\item[]
$F_{1,\alpha}(x,y) = [x_i^{\alpha}y_j^{1-\alpha},  \;\; i,j\in {\cal N}, i\not= j]$;

\item[]
$F_{2,\alpha}(x,y) = [(x_ix_j)^{\alpha}(y_iy_j)^{1-\alpha}, \;\; i,j\in {\cal N}, i\not= j]$;

\item[]
$F_{3,\alpha}(x,y) = [(x_iy_j)^{\alpha}(x_jy_i)^{1-\alpha},  \;\; i,j\in {\cal N}, i\not= j]$;

\item[]
$F_{4,\alpha}(x,y) = [\alpha x_ix_j + (1-\alpha)y_iy_j,  \;\; i,j\in {\cal N}, i\not= j]$;

\item[]
$F_{5,\alpha}(x,y) = [\alpha x_iy_j + (1-\alpha)x_jy_i,  \;\; i,j\in {\cal N}, i\not= j]$;

\item[]
$F_{6,\alpha}(x,y) = [(\alpha x_i+(1-\alpha)y_i)(\alpha x_j + (1-\alpha)y_j),  \;\; i,j\in {\cal N}, i\not= j]$;

\item[]
$F_{7,\alpha}(x,y) = [(\alpha x_i+(1-\alpha)y_i)(\alpha y_j + (1-\alpha)x_j),  \;\; i,j\in {\cal N}, i\not= j]$;

\item[]
$F_{8,\alpha,\beta}(x,y) = [(x_i^{\beta}y_i^{1-\beta})^\alpha (x_j^{\beta}y_j^{1-\beta})^{1-\alpha},  \;\; i,j\in {\cal N}, i\not= j]$;

\item[]
$F_{9,\alpha,\beta}(x,y) = [(x_i^{\beta}y_i^{1-\beta})^\alpha (y_j^{\beta}x_j^{1-\beta})^{1-\alpha},  \;\; i,j\in {\cal N}, i\not= j]$;

\item[]
$F_{10,\alpha,\beta}(x,y) = [\beta x_i^\alpha x_j^{1-\alpha} + (1-\beta)y_i^\alpha y_j^{1-\alpha},  \;\; i,j\in {\cal N}, i\not= j]$;

\item[]
$F_{11,\alpha,\beta}(x,y) = [\beta x_i^\alpha y_j^{1-\alpha} + (1-\beta)y_i^\alpha x_j^{1-\alpha},  \;\; i,j\in {\cal N}, i\not= j]$;

\item[]
$F_{12,\alpha,\beta}(x,y) = [(\beta x_i+(1-\beta)y_i)^\alpha(\beta x_j + (1-\beta)y_j)^{1-\alpha},  \;\; i,j\in {\cal N}, i\not= j]$;

\item[]
$F_{13,\alpha,\beta}(x,y) = [(\beta x_i+(1-\beta)y_i)^\alpha(\beta y_j + (1-\beta)x_j)^{1-\alpha},  \;\; i,j\in {\cal N}, i\not= j]$.
\end{itemize}
Here we assume
\begin{eqnarray*}
0^{\alpha}=
\begin{cases}
1, & if \; \alpha=0,\\
0, & if \; \alpha>0.
\end{cases}
\end{eqnarray*}

Accordingly, the set of the $G$-function pairs induced by $F_{\mu,\alpha}$ and $F_{\nu,\alpha,\beta}$ are respectively recorded as
$\mathcal{G}^F_{\mu,\alpha}$ and $\mathcal{G}^F_{\nu,\alpha,\beta}$, $\mu=1,\cdots,7$, $\nu=8,\cdots,13$.
\end{definition}

When $\alpha$ and $\beta$ take some special values, $\mathcal{G}^F_{\mu,\alpha}$ and
$\mathcal{G}^F_{\nu,\alpha,\beta}$ may be equal for different $\mu$ and $\nu$.

We give an important result, which shows that
$\mathcal{G}^F_{\mu,\alpha}$ and $\mathcal{G}^F_{\nu,\alpha,\beta}$ for any
$\mu=1,\cdots,7$, $\nu=8,\cdots,13$ and $\alpha, \beta \in [0 ,1]$ are not empty sets, but meaningful.

\begin{theorem}\label{thm3.4}
Suppose that $0 \le \alpha,\beta \le 1$, $\mathcal{G}^F_{\mu,\alpha}$ and
$\mathcal{G}^F_{\nu,\alpha,\beta}$ are defined by Definition 3.3, $\mu=1,\cdots,7$, $\nu=8,\cdots,13$.
Then $\mathcal{G}_n \times \mathcal{G}_n$ is a proper subset of
$\mathcal{G}^F_{\mu,\alpha}$ and $\mathcal{G}^F_{\nu,\alpha,\beta}$.
\end{theorem}

\begin{proof}
We first prove $\mathcal{G}_n\times \mathcal{G}_n \subseteq
\mathcal{G}^F_{\mu,\alpha}\cap\mathcal{G}^F_{\nu,\alpha,\beta}$, $\mu=1,\cdots,7$, $\nu=8,\cdots,13$.
In other words we need to prove that, for any $(g,h)\in \mathcal{G}_n \times \mathcal{G}_n$
and for each $\mu=1,\cdots,7$, $\nu=8,\cdots,13$, the condition
\begin{eqnarray} \label{eqnn3-2}
F_{\mu,\alpha}(|D(A)|e,|D(A)|e) > F_{\mu,\alpha}(g(A), h(A))
\end{eqnarray}
or
\begin{eqnarray} \label{eqnn3-3}
F_{\nu,\alpha,\beta}(|D(A)|e,|D(A)|e) > F_{\nu,\alpha,\beta}(g(A), h(A))
\end{eqnarray}
implies $A$ being nonsingular for any $A\in \mathbb{C}^{n\times n}$.

When $\mu=1$, by Theorem 3.2, we just need to prove
$(\tilde{r}^X, \tilde{r}^Y)\in \mathcal{G}^F_{1,\alpha}$,
for any $X,Y\in D^P_n$.

Without loss of generality, we assume that $A$ is irreducible.
Otherwise, by Lemma 2.1, we can prove that $A_{k,k}$ in the Frobenius normal form of $A$ is
nonsingular for $k = 1, \cdots, l$.

In this case (3.2) reduces to
\begin{eqnarray*}
F_{1,\alpha}(|D(A)|e,|D(A)|e) > F_{1,\alpha}(r^X(A), r^Y(A)),
\end{eqnarray*}
i.e.,
\begin{eqnarray}\label{eqnn3-5}
|a_{i,i}|^{\alpha}|a_{j,j}|^{1-\alpha}> [r_i^X(A)]^{\alpha}
[r^Y_j(A)]^{1-\alpha}, \; i,j\in {\cal N}, i\not= j.
\end{eqnarray}

We set $p_i= r_i^X(A)/|a_{i,i}|$ and $q_j= r_j^Y(A)/|a_{j,j}|$ for
all $i,j\in {\cal N}$. Then, from (3.4) and the
irreducibility of $A$, we have
\begin{eqnarray}\label{eqnn3-6}
0 < p_i^{\alpha}q_j^{1-\alpha} < 1, \; i,j\in {\cal N}, i\not= j.
\end{eqnarray}

Suppose $p_{i_1} \ge p_{i_2} \ge \cdots \ge p_{i_n}$ and
$q_{j_1} \ge q_{j_2} \ge \cdots \ge q_{j_n}$. Then $p_{i_n} > 0$
and $q_{j_n} > 0$.

If $p_{i_1}<1$ or $q_{j_1}<1$, then $A\in SDD$ and therefore $A$
is nonsingular.

If $p_{i_1} \ge 1$ and $q_{j_1} \ge 1$, then
$p_{i_1}^{\alpha}q_{j_1}^{(1-\alpha)} \ge 1$, so that, by
(3.5), $i_1=j_1$. Without loss of generality, suppose
$i_k=k$ for all $k\in {\cal N}$. Now, from
\begin{eqnarray*}
(p_1p_2)^{\alpha}(q_1q_{j_2})^{1-\alpha} =
(p_1^{\alpha}q_{j_2}^{1-\alpha}) (p_2^{\alpha}q_1^{1-\alpha}) < 1,
\end{eqnarray*}
it gets that
\begin{eqnarray*}
 \min\{p_1p_2, \; q_{1}q_{j_2}\}< 1.
\end{eqnarray*}

For the case when $p_1p_2 < 1$, let $D=diag[d, 1,\cdots,1]$
with $p_1 < d < 1/p_2$. Then $d > 1$, $D \in D^P_n$ and
\begin{eqnarray*}
\frac{r_1^{(XD)}(A)}{|a_{1,1}|} = d^{-1}\cdot\frac{r_1^{X}(A)}{|a_{1,1}|} = \frac{p_1}{d} < 1,
\end{eqnarray*}
and for $i \ge 2$,
\begin{eqnarray*}
\frac{r_i^{(XD)}(A)}{|a_{i,i}|} &=& \frac{1}{x_i}\left[d|a_{i,1}|x_1+\sum_{k\in
{\cal N}\setminus \{1,i\}}|a_{i,k}|x_k\right]\cdot \frac{1}{|a_{i,i}|}\\
& \le & d\cdot \frac{r_i^{X}(A)}{|a_{i,i}|} = dp_i \le dp_2 < 1.
\end{eqnarray*}
This has shown that $(XD)^{-1}A(XD) \in SDD$ and therefore $A$ is
nonsingular.

For the case when $q_{1}q_{j_2} < 1$ the proof is similar.

We have proved
$\mathcal{G}_n \times \mathcal{G}_n \subseteq \mathcal{G}^F_{1,\alpha}$.

When $\mu=2$, then (3.2) reduces to
\begin{eqnarray}\label{eqnn3-8}
|a_{i,i}||a_{j,j}| > [g_i(A)g_j(A)]^{\alpha}
[h_i(A)h_j(A)]^{1-\alpha}, \; i,j\in {\cal N}, i\not= j,
\end{eqnarray}
which can be rewritten as
\begin{eqnarray*}
|a_{i,i}||a_{j,j}| > [g_i^{\alpha}(A)h_i^{1-\alpha}(A)]
[g_j^{\alpha}(A)h_j^{1-\alpha}(A)], \; i,j\in {\cal N}, i\not= j.
\end{eqnarray*}

Let $\tilde{f} = g^{\alpha}h^{1-\alpha}$. Then we have $(\tilde{f},\tilde{f})\in \mathcal{G}_n \times \mathcal{G}_n$ and
\begin{eqnarray*}
|a_{i,i}||a_{j,j}| > \tilde{f}_i(A)\tilde{f}_j(A), \; i,j\in {\cal N}, i\not= j,
\end{eqnarray*}
which is equivalent to
\begin{eqnarray*}
F_{1,\frac{1}{2}}(|D(A)|e,|D(A)|e) > F_{1,\frac{1}{2}}(\tilde{f}(A), \tilde{f}(A)).
\end{eqnarray*}
Hence $A$ is nonsingular, so that
$\mathcal{G}_n \times \mathcal{G}_n \subseteq \mathcal{G}^F_{2,\alpha}$.

When $\mu=3$, then (3.2) reduces to
\begin{eqnarray*}
|a_{i,i}||a_{j,j}| > [g_i(A)h_j(A)]^{\alpha}
[g_j(A)h_i(A)]^{1-\alpha}, \; i,j\in {\cal N}, i\not= j,
\end{eqnarray*}
which can be rewritten as
\begin{eqnarray*}
|a_{i,i}||a_{j,j}| > [g_i^{\alpha}(A)h_i^{1-\alpha}(A)]
[h_j^{\alpha}(A)g_j^{1-\alpha}(A)], \; i,j\in {\cal N}, i\not= j.
\end{eqnarray*}
Let $\hat{f} = g^{\alpha}h^{1-\alpha}$ and $\check{f} = h^{\alpha}g^{1-\alpha}$. Then
 $\hat{f}, \check{f} \in \mathcal{G}_n$,
so that $(\hat{f},\check{f})\in \mathcal{G}_n \times \mathcal{G}_n$.
Now, we have
\begin{eqnarray*}
|a_{i,i}||a_{j,j}| > \hat{f}_i(A)\check{f}_j(A), \; i,j\in {\cal N}, i\not= j,
\end{eqnarray*}
which implies
\begin{eqnarray*}
F_{1,\frac{1}{2}}(|D(A)|e,|D(A)|e) > F_{1,\frac{1}{2}}(\hat{f}(A), \check{f}(A)).
\end{eqnarray*}
Hence $A$ is nonsingular, so that
$\mathcal{G}_n \times \mathcal{G}_n \subseteq \mathcal{G}^F_{3,\alpha}$.

When $\mu=4$, then (3.2) reduces to
\begin{eqnarray*}
|a_{i,i}||a_{j,j}| > \alpha g_i(A)g_j(A) + (1-\alpha)h_i(A)h_j(A), \; i,j\in {\cal N}, i\not= j.
\end{eqnarray*}
By Lemma 2.5, for $i,j\in {\cal N}$, we have
\begin{eqnarray*}
\alpha g_i(A)g_j(A) + (1-\alpha)h_i(A)h_j(A) \ge [g_i(A)g_j(A)]^{\alpha}[h_i(A)h_j(A)]^{1-\alpha},
\end{eqnarray*}
so that (3.6) holds. By the proof for $\mu=2$ we have proved that
$A$ is nonsingular, and therefore, $\mathcal{G}_n \times \mathcal{G}_n \subseteq \mathcal{G}^F_{4,\alpha}$.

Similarly, by Lemma 2.5 and the proofs for $\mu=2,3$ we can prove
$\mathcal{G}_n \times \mathcal{G}_n \subseteq \mathcal{G}^F_{\mu,\alpha}$, $\mu=5,6,7$.

When $\nu=8$, then (3.3) reduces to
\begin{eqnarray*}
|a_{i,i}|^\alpha|a_{j,j}|^{1-\alpha} >
[g_i^{\beta}(A)h_i^{1-\beta}(A)]^\alpha[g_j^{\beta}(A)h_j^{1-\beta}(A)]^{1-\alpha}, \;
 i,j\in {\cal N}, i\not= j.
\end{eqnarray*}
Let $\breve{f} = g^{\beta}h^{1-\beta}$. Then $\breve{f} \in \mathcal{G}_n$ and
\begin{eqnarray*}
|a_{i,i}|^\alpha|a_{j,j}|^{1-\alpha} >
\breve{f}_i^\alpha(A)\breve{f}_j^{1-\alpha}(A), \;\; i,j\in {\cal N}, i\not= j,
\end{eqnarray*}
i.e.,
\begin{eqnarray*}
F_{1,\alpha}(|D(A)|e,|D(A)|e) > F_{1,\alpha}(\breve{f}(A), \breve{f}(A)).
\end{eqnarray*}
Hence $A$ is nonsingular, so that
$\mathcal{G}_n \times \mathcal{G}_n \subseteq \mathcal{G}^F_{8,\alpha,\beta}$.

When $\nu=9$, then (3.3) reduces to
\begin{eqnarray*}
|a_{i,i}|^\alpha|a_{j,j}|^{1-\alpha} >
[g_i^{\beta}(A)h_i^{1-\beta}(A)]^\alpha[h_j^{\beta}(A)g_j^{1-\beta}(A)]^{1-\alpha}, \;
i,j\in {\cal N}, i\not= j.
\end{eqnarray*}
Let $\acute{f} = g^{\beta}h^{1-\beta}$, $\grave{f} = h^{\beta}g^{1-\beta}$.
Then $\acute{f}, \grave{f} \in \mathcal{G}_n$ and
\begin{eqnarray*}
|a_{i,i}|^\alpha|a_{j,j}|^{1-\alpha} >
\acute{f}_i^\alpha(A)\grave{f}_j^{1-\alpha}(A), \;\; i,j\in {\cal N}, i\not= j,
\end{eqnarray*}
i.e.,
\begin{eqnarray*}
F_{1,\alpha}(|D(A)|e,|D(A)|e) > F_{1,\alpha}(\acute{f}(A), \grave{f}(A)).
\end{eqnarray*}
Hence $A$ is nonsingular, so that
$\mathcal{G}_n \times \mathcal{G}_n \subseteq \mathcal{G}^F_{9,\alpha,\beta}$.

When $\nu=10$, then (3.3) reduces to
\begin{eqnarray*}
&& |a_{i,i}|^\alpha|a_{j,j}|^{1-\alpha} >
\beta g_i^\alpha(A)g_j(A)^{1-\alpha} + (1-\beta)h_i^\alpha(A)h_j^{1-\alpha}(A), \\
&& i,j\in {\cal N}, i\not= j.
\end{eqnarray*}
By Lemma 2.5, it gets that
\begin{eqnarray*}
|a_{i,i}|^\alpha|a_{j,j}|^{1-\alpha}
& > & [g_i^\alpha(A)g_j^{1-\alpha}(A)]^{\beta}[h_i^\alpha(A)h_j^{1-\alpha}(A)]^{1-\beta}\\
 & = & [g_i^{\beta}(A)h_i^{1-\beta}(A)]^\alpha[g_j^{\beta}(A)h_j^{1-\beta}(A)]^{1-\alpha}, \\
&& i,j\in {\cal N}, i\not= j.
\end{eqnarray*}
By the proof for $\nu = 8$ we have proved that $A$ is nonsingular, so that
$\mathcal{G}_n \times \mathcal{G}_n \subseteq \mathcal{G}^F_{10,\alpha,\beta}$.

Similarly, we can prove
$\mathcal{G}_n \times \mathcal{G}_n \subseteq \mathcal{G}^F_{11,\alpha,\beta}$.

When $\nu=12$, then (3.3) reduces to
\begin{eqnarray*}
|a_{i,i}|^\alpha|a_{j,j}|^{1-\alpha} &>&
[\beta g_i(A) + (1-\beta)h_i(A)]^\alpha[\beta g_j(A) + (1-\beta)h_j(A)]^{1-\alpha}, \\
&& i,j\in {\cal N}, i\not= j.
\end{eqnarray*}
Let $\bar{f} = \beta g + (1-\beta)h$. Then by Lemma 2.7 it gets that
$\bar{f} \in \mathcal{G}_n$ and
\begin{eqnarray*}
|a_{i,i}|^\alpha|a_{j,j}|^{1-\alpha} > [\bar{f}_i(A)]^{\alpha}[\bar{f}_j(A)]^{1-\alpha}, \; i,j\in {\cal N}, i\not= j,
\end{eqnarray*}
i.e.,
\begin{eqnarray*}
F_{1,\alpha}(|D(A)|e,|D(A)|e) > F_{1,\alpha}(\bar{f}(A), \bar{f}(A)).
\end{eqnarray*}
Hence $A$ is nonsingular, so that
$\mathcal{G}_n \times \mathcal{G}_n \subseteq \mathcal{G}^F_{12,\alpha,\beta}$.

Similarly, we can prove
$\mathcal{G}_n \times \mathcal{G}_n \subseteq \mathcal{G}^F_{13,\alpha,\beta}$.

Next, we prove that for each $\mu \in \{1,\cdots,7\}$ and $\nu \in \{8,\cdots,13\}$,
there exists $(g,h) \notin \mathcal{G}_n \times \mathcal{G}_n$,
but $(g,h) \in \mathcal{G}^F_{\mu,\alpha}$ or $(g,h) \in \mathcal{G}^F_{\nu,\alpha,\beta}$.

Let $r_\delta = \max_{k \in {\cal N}} r_k$.

We first consider the case when $0 < \alpha < 1$ for $\mu =1,2,4,6$ and
$0 < \beta < 1$ for $\nu = 8,10,12$. In this case
let $g^{(\alpha)} = (1+r_\delta)^{-1/\alpha}r$ and $h^{(\alpha)} = (1+r_\delta)^{1/(1-\alpha)}r$.
Then
\begin{eqnarray*}
&& F_{1,\alpha}(g^{(\alpha)},h^{(\alpha)}) = F_{1,\alpha}(r,r),\\
&& F_{k,\alpha}(g^{(\alpha)},h^{(\alpha)}) \ge F_{2,\alpha}(g^{(\alpha)},h^{(\alpha)})
= F_{2,\alpha}(r,r), \; k = 4,6.
\end{eqnarray*}
Hence, it is easy to see that $(g^{(\alpha)}, h^{(\alpha)}) \in \mathcal{G}^F_{\mu,\alpha}$,
$\mu =1,2,4,6$.

While, for $n=3$ let
\begin{eqnarray*}
A = \left[\begin{array}{ccc}
2  & -1 & -1\\
-1 &  2 & -1\\
-1 & -1 & 2
\end{array}\right].
\end{eqnarray*}
Then $A$ is singular. But we have
\begin{eqnarray*}
|a_{i,i}| = 2 > \frac{2}{3^\frac{1}{\alpha}} = g_i^{(\alpha)}(A), \; i=1,2,3.
\end{eqnarray*}
This shows that $g^{(\alpha)} \notin \mathcal{G}_n$, so that
$(g^{(\alpha)},h^{(\alpha)}) \notin \mathcal{G}_n \times \mathcal{G}_n$.

Completely same, for $0 < \beta < 1$,
let $g^{(\beta)} = (1+r_\delta)^{-1/\beta}r$ and $h^{(\beta)} = (1+r_\delta)^{1/(1-\beta)}r$.
Then $(g^{(\beta)},h^{(\beta)}) \notin \mathcal{G}_n \times \mathcal{G}_n$. While, from
\begin{eqnarray*}
&& F_{k,\alpha,\beta}(g^{(\beta)},h^{(\beta)}) \ge F_{8,\alpha,\beta}(g^{(\beta)},h^{(\beta)})
= F_{1,\alpha}(r,r), \; k =10,12,
\end{eqnarray*}
it is easy to see that $(g^{(\beta)},h^{(\beta)}) \in \mathcal{G}^F_{\nu,\alpha,\beta}$,
$\nu =8,10,12$.

Next, we consider the case when $\alpha = 1$ or $0$, $\beta = 1$ or $0$.
By [52, p.132, Remark 4], $\mathcal{G}_n$ is a proper subset of $\mathcal{F}_n$.
Hence there exists $h_0$ such that $h_0 \in \mathcal{F}_n$ and $h_0 \notin \mathcal{G}_n$.
Then for any $g \in \mathcal{G}_n$, $(g, h_0) \notin \mathcal{G}_n \times \mathcal{G}_n$ and
$(h_0, g) \notin \mathcal{G}_n \times \mathcal{G}_n$.
While, since
\begin{eqnarray*}
&& F_{1,1}(g,h_0) = F_{1,0}(h_0,g) = F_{1,1}(g,g),\\
&& F_{\mu,1}(g,h_0) = F_{\mu,0}(h_0,g) = F_{2,1}(g,g), \; \mu = 2,4,6,\\
&& F_{\nu,\alpha,1}(g,h_0) = F_{\nu,\alpha,0}(h_0,g) = F_{1,\alpha}(g,g), \; \nu = 8,10,12,
\end{eqnarray*}
it follows that $(g, h_0) \in \mathcal{G}^F_{\mu,1}, \mathcal{G}^F_{\nu,\alpha,1}$ and
$(h_0, g) \in \mathcal{G}^F_{\mu,0}, \mathcal{G}^F_{\nu,\alpha,0}$, $\mu = 1,2,4,6$, $\nu = 8,10,12$.

Let $\tilde{g} = (1+r_\delta)^{-1}r$ and $\tilde{h} = (1+r_\delta)r$. Then
for $0 \le \alpha,\beta \le 1$ we have
\begin{eqnarray*}
F_{7,\alpha}(\tilde{g},\tilde{h}) \ge
F_{3,\alpha}(\tilde{g},\tilde{h}) = F_{5,\alpha}(\tilde{g},\tilde{h})
= F_{2,1}(r,r).
\end{eqnarray*}
It is easy to prove that $(\tilde{g}, \tilde{h}) \in \mathcal{G}^F_{\mu,\alpha}$,
$\mu = 3,5,7$,
and $(\tilde{g},\tilde{h}) \notin \mathcal{G}_n \times \mathcal{G}_n$, because $\tilde{g} \notin \mathcal{G}_n$.

Up to now, we have prove that $\mathcal{G}_n \times \mathcal{G}_n$ is a proper subset of
$\mathcal{G}^F_{\mu,\alpha}$ and $\mathcal{G}^F_{\nu,\alpha,\beta}$ for $\mu=1,\cdots,7$, $\nu=8,10,12$.

At last, we consider the cases when $\nu=9,11,13$. When $\beta = 1$ then it is easy to prove that $\mathcal{G}^F_{9,\alpha,1} = \mathcal{G}^F_{11,\alpha,1} = \mathcal{G}^F_{13,\alpha,1} = \mathcal{G}^F_{1,\alpha}$.
This shows that $\mathcal{G}_n \times \mathcal{G}_n$ is a proper subset of
$\mathcal{G}^F_{\nu,\alpha,1}$, $\nu = 9,11,13$.

For the case when $\beta = 0$, then we have
$F_{9,\alpha,0}(x,y) = F_{11,\alpha,0}(x,y) = F_{13,\alpha,0}(x,y)
= [y_i^{\alpha}x_j^{1-\alpha},  \;\; i,j\in {\cal N}, i\not= j]= F_{1,\alpha}(y,x)$.
Completely similar to the proof for $\mu=1$ it can be proved that
$\mathcal{G}_n \times \mathcal{G}_n$ is a proper subset of
$\mathcal{G}^F_{\nu,\alpha,0}$, $\nu = 9,11,13$.

Now, we assume that $0 < \beta < 1$. Then $\alpha + \beta - 2\alpha\beta > 0$ and
$1 - \alpha - \beta + 2\alpha\beta > 0$. Let
\begin{eqnarray*}
\eta = \frac{\alpha + \beta - 2\alpha\beta}{1 - \alpha - \beta + 2\alpha\beta}.
\end{eqnarray*}
Then $\eta > 0$. Let $\hat{g} = (1+r_\delta)^{-\eta}r$. Then it is easy to prove that
$\hat{g} \notin \mathcal{G}_n$ so that $(\hat{g},\tilde{h}) \notin \mathcal{G}_n \times \mathcal{G}_n$.
While, since
\begin{eqnarray*}
F_{k,\alpha,\beta}(\hat{g},\tilde{h}) \ge F_{9,\alpha,\beta}(\hat{g},\tilde{h})
= F_{1,\alpha}(r,r), \; k = 11,13,
\end{eqnarray*}
then
$(\hat{g},\tilde{h}) \in \mathcal{G}^F_{\nu,\alpha,\beta}$, $\nu = 9,11,13$.
Hence $\mathcal{G}_n \times \mathcal{G}_n$ is a proper subset of
$\mathcal{G}^F_{\nu,\alpha,\beta}$, $\nu = 9,11,13$.
\end{proof}

From this theorem and Lemma 2.2, the following corollary is obvious.

\begin{corollary}\label{thm3.5}
For $\mu=1,\cdots,7$, $\nu=8,\cdots,13$ and $0 \le \alpha,\beta \le 1$, it gets that
$(g,h) \in \mathcal{G}^F_{\mu,\alpha}\cap\mathcal{G}^F_{\nu,\alpha,\beta}$, for
$g, h \in \{r,\tilde{r},c,\tilde{c},r^X,r^Y,c^X,c^Y,\tilde{r}^X,\tilde{r}^Y,\tilde{c}^X,\tilde{c}^Y\}$
and for any $X,Y\in D^P_n$.
\end{corollary}

\newpage \markboth{}{}
\section{Some criteria for generalized
strictly diagonally dominant matrices} \label{se4}
\markboth{}{\it 4\quad SOME CRITERIA FOR GDD}

In this section, using $G$-function pairs, we derive some known and new critical conditions about generalized
diagonally dominant matrices.

\begin{theorem}\label{thm4-1}
Let $A = [a_{i,j}] \in \mathbb{C}^{n\times n}$. Then $A \in GDD$
if and only if there exist $g, h \in \mathcal{G}_n$ and $\alpha,\beta \in [0,1]$
such that, for all $i,j\in {\cal N}$, $i\not= j$, one of the following conditions holds:

\begin{itemize}
\item[\rm(1)] $|a_{i,i}| > g_i(A)$;

\item[\rm(2)] $|a_{i,i}| > g_i^{\alpha}(A)h_i^{1-\alpha}(A)$;

\item[\rm(3)] $|a_{i,i}| > \alpha g_i(A) + (1-\alpha)h_i(A)$;

\item[\rm(4)] $|a_{i,i}||a_{j,j}| > g_i(A)g_j(A)$;

\item[\rm(5)] $|a_{i,i}||a_{j,j}| > g_i(A)h_j(A)$;

\item[\rm(6)] $|a_{i,i}||a_{j,j}| >
[g_i(A)g_j(A)]^{\alpha}[h_i(A)h_j(A)]^{1-\alpha}$;

\item[\rm(7)] $|a_{i,i}||a_{j,j}| >
[g_i(A)h_j(A)]^{\alpha}[g_j(A)h_i(A)]^{1-\alpha}$;

\item[\rm(8)] $|a_{i,i}||a_{j,j}| >
\alpha g_i(A)g_j(A) + (1-\alpha)h_i(A)h_j(A)$;

\item[\rm(9)] $|a_{i,i}||a_{j,j}| >
\alpha g_i(A)h_j(A) + (1-\alpha)g_j(A)h_i(A)$;

\item[\rm(10)]
$|a_{i,i}||a_{j,j}| >
 [\alpha g_i(A)+(1-\alpha)h_i(A)][\alpha g_j(A) + (1-\alpha)h_j(A)]$;

\item[\rm(11)]
$|a_{i,i}||a_{j,j}| >
 [\alpha g_i(A)+(1-\alpha)h_i(A)][\alpha h_j(A) + (1-\alpha)g_j(A)]$;

 \item[\rm(12)]
$|a_{i,i}|^{\alpha}|a_{j,j}|^{1-\alpha} > g_i^{\alpha}(A)g_j^{1-\alpha}(A)$;

 \item[\rm(13)]
$|a_{i,i}|^{\alpha}|a_{j,j}|^{1-\alpha} > g_i^{\alpha}(A)h_j^{1-\alpha}(A)$;

\item[\rm(14)]
$|a_{i,i}|^{\alpha}|a_{j,j}|^{1-\alpha} > [g_i^{\beta}(A) h_i^{1-\beta}(A)]^\alpha [g_j^{\beta}(A)h_j^{1-\beta}(A)]^{1-\alpha}$;

\item[\rm(15)]
$|a_{i,i}|^{\alpha}|a_{j,j}|^{1-\alpha} > [g_i^{\beta}(A)h_i^{1-\beta}(A)]^\alpha [h_j^{\beta}(A) g_j^{1-\beta}(A)]^{1-\alpha}$;

\item[\rm(16)]
$|a_{i,i}|^{\alpha}|a_{j,j}|^{1-\alpha} > \beta g_i^\alpha(A)g_j^{1-\alpha}(A) + (1-\beta)h_i^\alpha(A) h_j^{1-\alpha}(A)$;

\item[\rm(17)]
$|a_{i,i}|^{\alpha}|a_{j,j}|^{1-\alpha} > \beta g_i^\alpha(A) h_j^{1-\alpha}(A) + (1-\beta)h_i^\alpha(A) g_j^{1-\alpha}(A)$;

\item[\rm(18)]
$|a_{i,i}|^{\alpha}|a_{j,j}|^{1-\alpha} > [\beta g_i(A)+(1-\beta)h_i(A)]^\alpha[\beta g_j(A) + (1-\beta)h_j(A)]^{1-\alpha}$;

\item[\rm(19)]
$|a_{i,i}|^{\alpha}|a_{j,j}|^{1-\alpha} > [\beta g_i(A)+(1-\beta)h_i(A)]^\alpha[\beta h_j(A) + (1-\beta)g_j(A)]^{1-\alpha}$.
\end{itemize}
\end{theorem}

\begin{proof}
First we prove the necessity. Assume that $A \in GDD$. Then there
exist $X, Y\in D^P_n$ such that
 \begin{eqnarray*}
|a_{k,k}| > r_k^X(A), \;
|a_{k,k}| > c_k^Y(A), \; k \in {\cal N}.
\end{eqnarray*}
Let $g=r^X$ and $h=c^Y$. Then by Lemma 2.2, it follow that $g, h \in \mathcal{G}_n$.
The inequalities (1)-(19) are now easy to derive.

Next we prove the sufficiency. When one of (1)-(19) holds, then
by Theorem 3.1 we just need to prove
either there exist $\mu \in \{1,\cdots,7\}$ and $\alpha \in [0,1]$ such that
\begin{eqnarray}\label{eqn4.1}
F_{\mu,\alpha}(|D(A)|e,|D(A)|e) > F_{\mu,\alpha}(g(A),h(A))
\end{eqnarray}
or there exist $\nu \in \{8,\cdots,13\}$ and $\alpha,\beta \in [0,1]$ such that
\begin{eqnarray}\label{eqn4.2}
F_{\nu,\alpha,\beta}(|D(A)|e,|D(A)|e) > F_{\nu,\alpha,\beta}(g(A),h(A)).
\end{eqnarray}

For $\mu=1, 2$ and 3, let $\alpha = 1$. Then (4.1) reduces to (1), (4) and (5), respectively.

Similarly, for $\nu$ equals 8 and 12, let $\alpha = 1$ and $\beta = \alpha$. Then (4.2) reduces to (2) and (3), respectively.
While, for $\nu$ equals 10, set $\beta = 1$, then (4.2) reduces to (12).

Let $\mu=2,\cdots,7$, $\mu=1$ and $\nu =8,\cdots,13$. Then (4.1) is equivalent to
(6), $\cdots$, (11),(13), $\cdots$, (19), respectively.

The proof is complete.
\end{proof}

Clearly, in the theorem $g, h \in \mathcal{G}_n$ can be changed into $(g,h) \in \mathcal{G}^F_{\mu,\alpha}$
or $(g,h) \in \mathcal{G}^F_{\nu,\alpha,\beta}$ for the corresponding $\mu$ and $\nu$.

In [9, Lemma 3.1] the equivalence among (1), (4) and $H$-matrix is given.

The condition (2) for nonsingularity of $A$ is given by A. J. Hoffman [7, Theorem 1].

\begin{theorem}\label{thm4-2}
Let $A = [a_{i,j}] \in \mathbb{C}^{n\times n}$. Then $A \in GDD$
if and only if there exist $X, Y\in D^P_n$ and $\alpha,\beta \in [0,1]$
such that, for all $i,j\in {\cal N}$, $i\not= j$, one of the following conditions holds:

\begin{itemize}
\item[\rm(1)] $|a_{i,i}| > \tilde{r}_i^X(A)$;

\item[\rm(2)] $|a_{i,i}| > \tilde{c}_i^Y(A)$;

\item[\rm(3)] $|a_{i,i}| > [\tilde{r}_i^X(A)]^{\alpha}[\tilde{c}_i^Y(A)]^{1-\alpha}$;

\item[\rm(4)] $|a_{i,i}| > \alpha\tilde{r}_i^X(A) + (1-\alpha)\tilde{c}_i^Y(A)$;

\item[\rm(5)] $|a_{i,i}||a_{j,j}| > \tilde{r}_i^X(A)\tilde{r}_j^X(A)$;

\item[\rm(6)] $|a_{i,i}||a_{j,j}| > \tilde{c}_i^Y(A)\tilde{c}_j^Y(A)$;

\item[\rm(7)] $|a_{i,i}||a_{j,j}| > \tilde{r}_i^X(A)\tilde{c}_j^Y(A)$;

\item[\rm(8)] $|a_{i,i}||a_{j,j}| >
[\tilde{r}_i^X(A)\tilde{r}_j^X(A)]^{\alpha}[\tilde{c}_i^Y(A)\tilde{c}_j^Y(A)]^{1-\alpha}$;

\item[\rm(9)] $|a_{i,i}||a_{j,j}| >
[\tilde{r}_i^X(A)\tilde{c}_j^Y(A)]^{\alpha}[\tilde{r}_j^X(A)\tilde{c}_i^Y(A)]^{1-\alpha}$;

\item[\rm(10)] $|a_{i,i}||a_{j,j}| >
\alpha\tilde{r}_i^X(A)\tilde{r}_j^X(A) + (1-\alpha)\tilde{c}_i^Y(A)\tilde{c}_j^Y(A)$;

\item[\rm(11)] $|a_{i,i}||a_{j,j}| >
\alpha\tilde{r}_i^X(A)\tilde{c}_j^Y(A) + (1-\alpha)\tilde{r}_j^X(A)\tilde{c}_i^Y(A)$;

\item[\rm(12)]
$|a_{i,i}||a_{j,j}| >
 [\alpha\tilde{r}_i^X(A)+(1-\alpha)\tilde{c}_i^Y(A)][\alpha\tilde{r}_j^X(A) + (1-\alpha)\tilde{c}_j^Y(A)]$;

\item[\rm(13)]
$|a_{i,i}||a_{j,j}| >
 [\alpha\tilde{r}_i^X(A)+(1-\alpha)\tilde{c}_i^Y(A)][\alpha\tilde{c}_j^Y(A) + (1-\alpha)\tilde{r}_j^X(A)]$;

 \item[\rm(14)]
$|a_{i,i}|^{\alpha}|a_{j,j}|^{1-\alpha} > [\tilde{r}_i^X(A)]^{\alpha}[\tilde{r}_j^X(A)]^{1-\alpha}$;

\item[\rm(15)] $|a_{i,i}|^{\alpha}|a_{j,j}|^{1-\alpha} > [\tilde{c}_i^Y(A)]^{\alpha}[\tilde{c}_j^Y(A)]^{1-\alpha}$;

\item[\rm(16)] $|a_{i,i}|^{\alpha}|a_{j,j}|^{1-\alpha} >
[\tilde{r}_i^X(A)]^{\alpha}[\tilde{c}_j^Y(A)]^{1-\alpha}$;

\item[\rm(17)]
$|a_{i,i}|^{\alpha}|a_{j,j}|^{1-\alpha} > ([\tilde{r}_i^X(A)]^{\beta} [\tilde{c}_i^Y(A)]^{1-\beta})^\alpha ([\tilde{r}_j^X(A)]^{\beta}[\tilde{c}_j^Y(A)]^{1-\beta})^{1-\alpha}$;

\item[\rm(18)]
$|a_{i,i}|^{\alpha}|a_{j,j}|^{1-\alpha} > ([\tilde{r}_i^X(A)]^{\beta}[\tilde{c}_i^Y(A)]^{1-\beta})^\alpha ([\tilde{c}_j^Y(A)]^{\beta} [\tilde{r}_j^X(A)]^{1-\beta})^{1-\alpha}$;

\item[\rm(19)]
$|a_{i,i}|^{\alpha}|a_{j,j}|^{1-\alpha} > \beta [\tilde{r}_i^X(A)]^{\alpha}[\tilde{r}_j^X(A)]^{1-\alpha} + (1-\beta)[\tilde{c}_i^Y(A)]^{\alpha} [\tilde{c}_j^Y(A)]^{1-\alpha}$;

\item[\rm(20)]
$|a_{i,i}|^{\alpha}|a_{j,j}|^{1-\alpha} > \beta [\tilde{r}_i^X(A)]^{\alpha} [\tilde{c}_j^Y(A)]^{1-\alpha} + (1-\beta)[\tilde{c}_i^Y(A)]^{\alpha} [\tilde{r}_j^X(A)]^{1-\alpha}$;

\item[\rm(21)]
$|a_{i,i}|^{\alpha}|a_{j,j}|^{1-\alpha} > [\beta \tilde{r}_i^X(A)+(1-\beta)\tilde{c}_i^Y(A)]^\alpha[\beta \tilde{r}_j^X(A) + (1-\beta)\tilde{c}_j^Y(A)]^{1-\alpha}$;

\item[\rm(22)]
$|a_{i,i}|^{\alpha}|a_{j,j}|^{1-\alpha} > [\beta \tilde{r}_i^X(A)+(1-\beta)\tilde{c}_i^Y(A)]^\alpha[\beta \tilde{c}_j^Y(A) + (1-\beta)\tilde{r}_j^X(A)]^{1-\alpha}$.
\end{itemize}
\end{theorem}

\begin{proof}
First we prove the necessity. Assume that $A \in GDD$. Then there
exist $X, Y\in D^P_n$ such that
 \begin{eqnarray*}
|a_{k,k}| > r_k^X(A), \;
|a_{k,k}| > c_k^Y(A), \; k \in {\cal N},
\end{eqnarray*}
so that
 \begin{eqnarray*}
|a_{k,k}| > \tilde{r}_k^X(A), \;
|a_{k,k}| > \tilde{c}_k^Y(A), \; k \in {\cal N},
\end{eqnarray*}
since $r^X(A) \ge \tilde{r}^X(A)$ and $c^Y(A) \ge \tilde{c}^Y(A)$.
Hence, (1)-(22) are easy to derive.

Now, using Theorem 4.1, we prove the sufficiency.

Let $g = \tilde{c}^Y$. Then from (1), (4) and (12) of Theorem 4.1
we derive (2), (6) and (15), respectively.

Let $(g,h)=(\tilde{r}^X,\tilde{c}^Y)$. Then from (1)-(19) of Theorem 4.1
we derive (1), (3), (4), (5), (7)-(14) and (16)-(22), respectively.

The proof is complete.
\end{proof}

We should point out that $X$, $Y$, $\alpha$ and $\beta$ in (1) to (22) of the theorem do not need to be the same.

In each of (1) to (22), since $X$ and $Y$ can be two different matrices, all the results given by this theorem are new and better.

Since $r^X \ge \tilde{r}^X$ and $c^Y \ge \tilde{c}^Y$, from the proof of Theorems 4.2, the following theorem is immediately.

\begin{theorem}\label{thm4-3}
Let $A = [a_{i,j}] \in \mathbb{C}^{n\times n}$. Then $A \in GDD$
if and only if there exist $X, Y\in D^P_n$ and $\alpha,\beta \in [0,1]$
such that, for all $i,j\in {\cal N}$, $i\not= j$, one of the following conditions holds:

\begin{itemize}
\item[\rm(1)] $|a_{i,i}| > r_i^X(A)$, $i\in {\cal N}$;

\item[\rm(2)] $|a_{i,i}| > c_i^Y(A)$, $i\in {\cal N}$;

\item[\rm(3)] $|a_{i,i}| > [r_i^X(A)]^{\alpha}[c_i^Y(A)]^{1-\alpha}$;

\item[\rm(4)] $|a_{i,i}| > \alpha r_i^X(A) + (1-\alpha)c_i^Y(A)$;

\item[\rm(5)] $|a_{i,i}||a_{j,j}| > r_i^X(A)r_j^X(A)$, $i,j\in {\cal N}$, $i\not= j$;

\item[\rm(6)] $|a_{i,i}||a_{j,j}| > c_i^Y(A)c_j^Y(A)$, $i,j\in {\cal N}$, $i\not= j$;

\item[\rm(7)] $|a_{i,i}||a_{j,j}| > r_i^X(A)c_j^Y(A)$, $i,j\in {\cal N}$, $i\not= j$;

\item[\rm(8)] $|a_{i,i}||a_{j,j}| >
[r_i^X(A)r_j^X(A)]^{\alpha}[c_i^Y(A)c_j^Y(A)]^{1-\alpha}$;

\item[\rm(9)] $|a_{i,i}||a_{j,j}| >
[r_i^X(A)c_j^Y(A)]^{\alpha}[r_j^X(A)c_i^Y(A)]^{1-\alpha}$;

\item[\rm(10)] $|a_{i,i}||a_{j,j}| >
\alpha r_i^X(A)r_j^X(A) + (1-\alpha)c_i^Y(A)c_j^Y(A)$;

\item[\rm(11)] $|a_{i,i}||a_{j,j}| >
\alpha r_i^X(A)c_j^Y(A) + (1-\alpha)r_j^X(A)c_i^Y(A)$;

\item[\rm(12)] $|a_{i,i}||a_{j,j}| >
[\alpha r_i^X(A) + (1-\alpha)c_i^Y(A)][\alpha r_j^X(A) + (1-\alpha)c_j^Y(A)]$;

\item[\rm(13)] $|a_{i,i}||a_{j,j}| >
[\alpha r_i^X(A) + (1-\alpha)c_i^Y(A)][\alpha c_j^Y(A) + (1-\alpha)r_j^X(A)]$;

 \item[\rm(14)]
$|a_{i,i}|^{\alpha}|a_{j,j}|^{1-\alpha} > [r_i^X(A)]^{\alpha}[r_j^Y(A)]^{1-\alpha}$;

\item[\rm(15)] $|a_{i,i}|^{\alpha}|a_{j,j}|^{1-\alpha} > [c_i^X(A)]^{\alpha}[c_j^Y(A)]^{1-\alpha}$;

\item[\rm(16)] $|a_{i,i}|^{\alpha}|a_{j,j}|^{1-\alpha} >
[r_i^X(A)]^{\alpha}[c_j^Y(A)]^{1-\alpha}$;

\item[\rm(17)]
$|a_{i,i}|^{\alpha}|a_{j,j}|^{1-\alpha} > ([r_i^X(A)]^{\beta} [c_i^Y(A)]^{1-\beta})^\alpha ([r_j^X(A)]^{\beta}[c_j^Y(A)]^{1-\beta})^{1-\alpha}$;

\item[\rm(18)]
$|a_{i,i}|^{\alpha}|a_{j,j}|^{1-\alpha} > ([r_i^X(A)]^{\beta}[c_i^Y(A)]^{1-\beta})^\alpha ([c_j^Y(A)]^{\beta}[r_j^X(A)]^{1-\beta})^{1-\alpha}$;

\item[\rm(19)]
$|a_{i,i}|^{\alpha}|a_{j,j}|^{1-\alpha} > \beta [r_i^X(A)]^{\alpha}[r_j^X(A)]^{1-\alpha} + (1-\beta)[c_i^Y(A)]^{\alpha} [c_j^Y(A)]^{1-\alpha}$;

\item[\rm(20)]
$|a_{i,i}|^{\alpha}|a_{j,j}|^{1-\alpha} > \beta [r_i^X(A)]^{\alpha} [c_j^Y(A)]^{1-\alpha} + (1-\beta)[c_i^Y(A)]^{\alpha} [r_j^X(A)]^{1-\alpha}$;

\item[\rm(21)]
$|a_{i,i}|^{\alpha}|a_{j,j}|^{1-\alpha} > [\beta r_i^X(A)+(1-\beta)c_i^Y(A)]^\alpha[\beta r_j^X(A) + (1-\beta)c_j^Y(A)]^{1-\alpha}$;

\item[\rm(22)]
$|a_{i,i}|^{\alpha}|a_{j,j}|^{1-\alpha} > [\beta r_i^X(A)+(1-\beta)c_i^Y(A)]^\alpha[\beta c_j^Y(A) + (1-\beta)r_j^X(A)]^{1-\alpha}$.
\end{itemize}
\end{theorem}

Clearly, we can take $(g,h)$ in Theorem 4.1 as some other combinations, just
$g,h \in \{r^X,r^Y,c^X,c^Y,\tilde{r}^X,\tilde{r}^Y,\tilde{c}^X,\tilde{c}^Y\}$
for any $X,Y\in D^P_n$,
then we can export some results those are exactly like Theorems 4.2 and 4.3.

By using Theorems 4.1, 4.2 and 4.3 we can construct the
corresponding equivalent conditions of $M$-matrices.

Now, we give some sufficient conditions for generalized diagonally dominant matrices.

\begin{theorem}\label{thm4-4}
Let $A = [a_{i,j}] \in \mathbb{C}^{n\times n}$.
Then $A \in GDD$, provided there exist $g, h\in \mathcal{G}_n$ and $\alpha,\beta \in [0,1]$
such that, for all $i,j\in {\cal N}$, $i\not= j$, one of the following conditions holds:

\begin{itemize}
\item[\rm(1)]
$|a_{i,i}||a_{j,j}| >
 [\alpha g_i(A)+(1-\alpha)h_j(A)][\alpha g_j(A) + (1-\alpha)h_i(A)]$;

\item[\rm(2)]
$|a_{i,i}||a_{j,j}| >
 [\alpha g_i(A)+(1-\alpha)g_j(A)][\alpha h_j(A) + (1-\alpha)h_i(A)]$;

\item[\rm(3)]
$|a_{i,i}|^{\alpha}|a_{j,j}|^{1-\alpha} >
 \alpha g_i(A)+(1-\alpha)g_j(A)$;

\item[\rm(4)]
$|a_{i,i}|^{\alpha}|a_{j,j}|^{1-\alpha} >
\alpha g_i(A)+(1-\alpha)h_j(A)$;

 \item[\rm(5)]
$|a_{i,i}|^\alpha|a_{j,j}|^{1-\alpha} >
 \alpha g_i^{\beta}(A)h_i^{1-\beta}(A) + (1-\alpha)g_j^{\beta}(A)h_j^{1-\beta}(A)$;

\item[\rm(6)]
$|a_{i,i}|^\alpha|a_{j,j}|^{1-\alpha} >
 \alpha g_i^{\beta}(A)h_i^{1-\beta}(A) + (1-\alpha)h_j^{\beta}(A)g_j^{1-\beta}(A)$;

\item[\rm(7)]
$|a_{i,i}|^\alpha|a_{j,j}|^{1-\alpha} >
[\alpha g_i(A)+(1-\alpha)h_j(A)]^{\beta}[\alpha h_i(A) + (1-\alpha)g_j(A)]^{1-\beta}$;

\item[\rm(8)]
$|a_{i,i}|^\alpha|a_{j,j}|^{1-\alpha} >
[\alpha g_i(A)+(1-\alpha)g_j(A)]^{\beta}[\alpha h_i(A) + (1-\alpha)h_j(A)]^{1-\beta}$.
\end{itemize}
\end{theorem}

\begin{proof}
If (1) holds, then it follows by Lemma 2.5 that
\begin{eqnarray*}
|a_{i,i}||a_{j,j}| & > &
g_i^{\alpha}(A)h_j^{1-\alpha}(A)g_j^{\alpha}(A)h_i^{1-\alpha}(A)\\
& = & [g_i(A)g_j(A)]^{\alpha}[h_i(A)h_j(A)]^{1-\alpha},
\;\; i,j\in {\cal N}, i\not= j,
\end{eqnarray*}
i.e., the condition (6) in Theorem 4.1 holds, so that $A \in GDD$.

Similarly, when (2) holds it can be proved that the condition (7)
in Theorem 4.1 holds, so that $A \in GDD$.

When (3) holds, then
\begin{eqnarray*}
|a_{i,i}|^{\alpha}|a_{j,j}|^{1-\alpha} >
g_i^{\alpha}(A)g_j^{1-\alpha}(A), \;
 i,j\in {\cal N}, i\not= j,
\end{eqnarray*}
which implies that (12) of Theorem 4.1 holds,
so that $A \in GDD$.

Similarly, when (4) holds, we can prove (13) of Theorem 4.1 to be valid,
so that $A \in GDD$.

If (5) holds, then
\begin{eqnarray*}
|a_{i,i}|^\alpha|a_{j,j}|^{1-\alpha} >
 [g_i^{\beta}(A)h_i^{1-\beta}(A)]^\alpha [g_j^{\beta}(A)h_j^{1-\beta}(A)]^{1-\alpha},
  i,j\in {\cal N}, i\not= j,
\end{eqnarray*}
i.e., (14) of Theorem 4.1 holds,
so that $A \in GDD$.

Similarly, when (6) holds, we get (15) of Theorem 4.1.
Hence $A \in GDD$.

If (7) holds, then
\begin{eqnarray*}
|a_{i,i}|^\alpha|a_{j,j}|^{1-\alpha} & > &
 [g_i^\alpha h_j^{1-\alpha}(A)]^{\beta}[h_i^\alpha(A)g_j^{1-\alpha}(A)]^{1-\beta}\\
& = & [g_i^{\beta}(A)h_i^{1-\beta}(A)]^\alpha [h_j^{\beta}(A)g_j^{1-\beta}(A)]^{1-\alpha},
 i,j\in {\cal N}, i\not= j,
\end{eqnarray*}
i.e., (15) of Theorem 4.1 holds,
so that $A \in GDD$.

Similarly, when (8) holds, it can be derived (14) of Theorem 4.1,
so that $A \in GDD$.
\end{proof}

Set $(g,h) = (\tilde{r}^X, \tilde{c}^Y)$ or $(g, h) = (r^X, c^Y)$ for $X, Y\in D^P_n$. Then the following
results are direct.

\begin{theorem}\label{thm4-5}
Let $A = [a_{i,j}] \in \mathbb{C}^{n\times n}$.
Then $A \in GDD$, provided there exist $X, Y\in D^P_n$ and $\alpha,\beta \in [0,1]$
such that, for all $i,j\in {\cal N}$, $i\not= j$, one of the following conditions holds:

\begin{itemize}
\item[\rm(1)]
$|a_{i,i}||a_{j,j}| >
 [\alpha \tilde{r}_i^X(A)+(1-\alpha)\tilde{c}_j^Y(A)][\alpha \tilde{r}_j^X(A) + (1-\alpha)\tilde{c}_i^Y(A)]$;

\item[\rm(2)]
$|a_{i,i}||a_{j,j}| >
 [\alpha \tilde{r}_i^X(A)+(1-\alpha)\tilde{r}_j^X(A)][\alpha \tilde{c}_j^Y(A) + (1-\alpha)\tilde{c}_i^Y(A)]$;

\item[\rm(3)]
$|a_{i,i}|^{\alpha}|a_{j,j}|^{1-\alpha} >
\alpha \tilde{r}_i^X(A)+(1-\alpha)\tilde{r}_j^X(A)$;

\item[\rm(4)]
$|a_{i,i}|^\alpha|a_{j,j}|^{1-\alpha} >
\alpha \tilde{c}_i^Y(A)+(1-\alpha)\tilde{c}_j^Y(A)$;

\item[\rm(5)]
$|a_{i,i}|^{\alpha}|a_{j,j}|^{1-\alpha} >
\alpha \tilde{r}_i^X(A)+(1-\alpha)\tilde{c}_j^Y(A)$;

\item[\rm(6)]
$|a_{i,i}|^\alpha|a_{j,j}|^{1-\alpha} >
 \alpha [\tilde{r}_i^X(A)]^{\beta}[\tilde{c}_i^Y(A)]^{1-\beta} + (1-\alpha)[\tilde{r}_j^X(A)]^{\beta}[\tilde{c}_j^Y(A)]^{1-\beta}$;

\item[\rm(7)]
$|a_{i,i}|^\alpha|a_{j,j}|^{1-\alpha} >
 \alpha [\tilde{r}_i^X(A)]^{\beta}[\tilde{c}_i^Y(A)]^{1-\beta} + (1-\alpha)[\tilde{c}_j^Y(A)]^{\beta}[\tilde{r}_j^X(A)]^{1-\beta}$;

\item[\rm(8)]
$|a_{i,i}|^\alpha|a_{j,j}|^{1-\alpha} >
[\alpha \tilde{r}_i^X(A)+(1-\alpha)\tilde{c}_j^Y(A)]^{\beta}[\alpha \tilde{c}_i^Y(A) + (1-\alpha)\tilde{r}_j^X(A)]^{1-\beta}$;

\item[\rm(9)]
$|a_{i,i}|^\alpha|a_{j,j}|^{1-\alpha} >
[\alpha \tilde{r}_i^X(A)+(1-\alpha)\tilde{r}_j^X(A)]^{\beta}[\alpha \tilde{c}_i^Y(A) + (1-\alpha)\tilde{c}_j^Y(A)]^{1-\beta}$
\end{itemize}

and

\begin{itemize}
\item[\rm(1')]
$|a_{i,i}||a_{j,j}| >
 [\alpha r_i^X(A)+(1-\alpha)c_j^Y(A)][\alpha r_j^X(A) + (1-\alpha)c_i^Y(A)]$;

\item[\rm(2')]
$|a_{i,i}||a_{j,j}| >
 [\alpha r_i^X(A)+(1-\alpha)r_j^X(A)][\alpha c_j^Y(A) + (1-\alpha)c_i^Y(A)]$;

\item[\rm(3')]
$|a_{i,i}|^{\alpha}|a_{j,j}|^{1-\alpha} >
\alpha r_i^X(A)+(1-\alpha)r_j^X(A)$;

\item[\rm(4')]
$|a_{i,i}|^{\alpha}|a_{j,j}|^{1-\alpha} >
\alpha c_i^Y(A)+(1-\alpha)c_j^Y(A)$;

\item[\rm(5')]
$|a_{i,i}|^{\alpha}|a_{j,j}|^{1-\alpha} >
\alpha r_i^X(A)+(1-\alpha)c_j^Y(A)$;

\item[\rm(6')]
$|a_{i,i}|^\alpha|a_{j,j}|^{1-\alpha} >
 \alpha [r_i^X(A)]^{\beta}[c_i^Y(A)]^{1-\beta} + (1-\alpha)[r_j^X(A)]^{\beta}[c_j^Y(A)]^{1-\beta}$;

\item[\rm(7')]
$|a_{i,i}|^\alpha|a_{j,j}|^{1-\alpha} >
 \alpha [r_i^X(A)]^{\beta}[c_i^Y(A)]^{1-\beta} + (1-\alpha)[c_j^Y(A)]^{\beta}[r_j^X(A)]^{1-\beta}$;

\item[\rm(8')]
$|a_{i,i}|^\alpha|a_{j,j}|^{1-\alpha} >
[\alpha r_i^X(A)+(1-\alpha)c_j^Y(A)]^{\beta}[\alpha c_i^Y(A) + (1-\alpha)r_j^X(A)]^{1-\beta}$;

\item[\rm(9')]
$|a_{i,i}|^\alpha|a_{j,j}|^{1-\alpha} >
[\alpha r_i^X(A)+(1-\alpha)r_j^X(A)]^{\beta}[\alpha c_i^Y(A) + (1-\alpha)c_j^Y(A)]^{1-\beta}$.
\end{itemize}
\end{theorem}

In Theorems 4.2, 4.3 and 4.5, if $X = Y = I$, then
we can directly derive the following sufficient conditions for generalized diagonally dominant matrices.

\begin{theorem}\label{thm4-6}
Let $A = [a_{i,j}] \in \mathbb{C}^{n\times n}$.
Then $A \in GDD$, provided there exist $\alpha,\beta \in [0,1]$
such that, for all $i,j\in {\cal N}$, $i\not= j$, one of the following conditions holds:

\begin{itemize}
\item[\rm(1)] $|a_{i,i}| > \tilde{r}_i(A)$;

\item[\rm(2)] $|a_{i,i}| > \tilde{c}_i(A)$;

\item[\rm(3)] $|a_{i,i}| > \tilde{r}_i^{\alpha}(A)\tilde{c}_i^{1-\alpha}(A)$;

\item[\rm(4)] $|a_{i,i}| > \alpha \tilde{r}_i(A) + (1-\alpha)\tilde{c}_i(A)$;

\item[\rm(5)] $|a_{i,i}||a_{j,j}| > \tilde{r}_i(A)\tilde{r}_j(A)$;

\item[\rm(6)] $|a_{i,i}||a_{j,j}| > \tilde{c}_i(A)\tilde{c}_j(A)$;

\item[\rm(7)] $|a_{i,i}||a_{j,j}| > \tilde{r}_i(A)\tilde{c}_j(A)$;

\item[\rm(8)] $|a_{i,i}||a_{j,j}| >
[\tilde{r}_i(A)\tilde{r}_j(A)]^{\alpha}[\tilde{c}_i(A)\tilde{c}_j(A)]^{1-\alpha}$;

\item[\rm(9)] $|a_{i,i}||a_{j,j}| >
[\tilde{r}_i(A)\tilde{c}_j(A)]^{\alpha}[\tilde{r}_j(A)\tilde{c}_i(A)]^{1-\alpha}$;

\item[\rm(10)] $|a_{i,i}||a_{j,j}| >
\alpha\tilde{r}_i(A)\tilde{r}_j(A) + (1-\alpha)\tilde{c}_i(A)\tilde{c}_j(A)$;

\item[\rm(11)] $|a_{i,i}||a_{j,j}| >
\alpha\tilde{r}_i(A)\tilde{c}_j(A) + (1-\alpha)\tilde{r}_j(A)\tilde{c}_i(A)$;

\item[\rm(12)]
$|a_{i,i}||a_{j,j}| >
 [\alpha\tilde{r}_i(A)+(1-\alpha)\tilde{c}_i(A)][\alpha\tilde{r}_j(A) + (1-\alpha)\tilde{c}_j(A)]$;

\item[\rm(13)]
$|a_{i,i}||a_{j,j}| >
 [\alpha\tilde{r}_i(A)+(1-\alpha)\tilde{c}_i(A)][\alpha\tilde{c}_j(A) + (1-\alpha)\tilde{r}_j(A)]$;

 \item[\rm(14)]
$|a_{i,i}||a_{j,j}| >
 [\alpha\tilde{r}_i(A)+(1-\alpha)\tilde{c}_j(A)][\alpha\tilde{r}_j(A) + (1-\alpha)\tilde{c}_i(A)]$;

\item[\rm(15)]
$|a_{i,i}||a_{j,j}| >
 [\alpha\tilde{r}_i(A)+(1-\alpha)\tilde{r}_j(A)][\alpha\tilde{c}_j(A) + (1-\alpha)\tilde{c}_i(A)]$;

 \item[\rm(16)]
$|a_{i,i}|^{\alpha}|a_{j,j}|^{1-\alpha} > \tilde{r}_i^{\alpha}(A)\tilde{r}_j^{1-\alpha}(A)$;

\item[\rm(17)] $|a_{i,i}|^{\alpha}|a_{j,j}|^{1-\alpha} > \tilde{c}_i^{\alpha}(A)\tilde{c}_j^{1-\alpha}(A)$;

\item[\rm(18)] $|a_{i,i}|^{\alpha}|a_{j,j}|^{1-\alpha} >
\tilde{r}_i^{\alpha}(A)\tilde{c}_j^{1-\alpha}(A)$;

\item[\rm(19)]
$|a_{i,i}|^\alpha|a_{j,j}|^{1-\alpha} >
\alpha \tilde{r}_i(A)+(1-\alpha)\tilde{r}_j(A)$;

\item[\rm(20)]
$|a_{i,i}|^{\alpha}|a_{j,j}|^{1-\alpha} >
\alpha \tilde{c}_i(A)+(1-\alpha)\tilde{c}_j(A)$;

\item[\rm(21)]
$|a_{i,i}|^{\alpha}|a_{j,j}|^{1-\alpha} >
\alpha \tilde{r}_i(A)+(1-\alpha)\tilde{c}_j(A)$;

\item[\rm(22)]
$|a_{i,i}|^{\alpha}|a_{j,j}|^{1-\alpha} > [\tilde{r}_i^{\beta}(A)\tilde{c}_i^{1-\beta}(A)]^\alpha [\tilde{r}_j^{\beta}(A)\tilde{c}_j^{1-\beta}(A)]^{1-\alpha}$;

\item[\rm(23)]
$|a_{i,i}|^{\alpha}|a_{j,j}|^{1-\alpha} > [\tilde{r}_i^{\beta}(A)\tilde{c}_i^{1-\beta}(A)]^\alpha [\tilde{c}_j^{\beta}(A)\tilde{r}_j^{1-\beta}(A)]^{1-\alpha}$;

\item[\rm(24)]
$|a_{i,i}|^{\alpha}|a_{j,j}|^{1-\alpha} > \beta\tilde{r}_i^{\alpha}(A)\tilde{r}_j^{1-\alpha}(A) + (1-\beta)\tilde{c}_i^{\alpha}(A)\tilde{c}_j^{1-\alpha}(A)$;

\item[\rm(25)]
$|a_{i,i}|^{\alpha}|a_{j,j}|^{1-\alpha} > \beta\tilde{r}_i^{\alpha}(A)\tilde{c}_j^{1-\alpha}(A) + (1-\beta)\tilde{c}_i^{\alpha}(A)\tilde{r}_j^{1-\alpha}(A)$;

\item[\rm(26)]
$|a_{i,i}|^\alpha|a_{j,j}|^{1-\alpha} >
 \alpha\tilde{r}_i^{\beta}(A)\tilde{c}_i^{1-\beta}(A) + (1-\alpha)\tilde{r}_j^{\beta}(A)\tilde{c}_j^{1-\beta}(A)$;

\item[\rm(27)]
$|a_{i,i}|^\alpha|a_{j,j}|^{1-\alpha} >
 \alpha\tilde{r}_i^{\beta}(A)\tilde{c}_i^{1-\beta}(A) + (1-\alpha)\tilde{c}_j^{\beta}(A)\tilde{r}_j^{1-\beta}(A)$;

\item[\rm(28)]
$|a_{i,i}|^{\alpha}|a_{j,j}|^{1-\alpha} > [\beta\tilde{r}_i(A)+(1-\beta)\tilde{c}_i(A)]^\alpha[\beta\tilde{r}_j(A) + (1-\beta)\tilde{c}_j(A)]^{1-\alpha}$;

\item[\rm(29)]
$|a_{i,i}|^{\alpha}|a_{j,j}|^{1-\alpha} > [\beta \tilde{r}_i(A)+(1-\beta)\tilde{c}_i(A)]^\alpha[\beta \tilde{c}_j(A) + (1-\beta)\tilde{r}_j(A)]^{1-\alpha}$;

\item[\rm(30)]
$|a_{i,i}|^\alpha|a_{j,j}|^{1-\alpha} >
[\alpha\tilde{r}_i(A)+(1-\alpha)\tilde{c}_j(A)]^{\beta}[\alpha\tilde{c}_i(A) + (1-\alpha)\tilde{r}_j(A)]^{1-\beta}$;

\item[\rm(31)]
$|a_{i,i}|^\alpha|a_{j,j}|^{1-\alpha} >
[\alpha\tilde{r}_i(A)+(1-\alpha)\tilde{r}_j(A)]^{\beta}[\alpha\tilde{c}_i(A) + (1-\alpha)\tilde{c}_j(A)]^{1-\beta}$.
\end{itemize}
\end{theorem}

\begin{theorem}\label{thm4-7}
Let $A = [a_{i,j}] \in \mathbb{C}^{n\times n}$.
Then $A \in GDD$, provided there exist $\alpha,\beta \in [0,1]$
such that, for all $i,j\in {\cal N}$, $i\not= j$, one of the following conditions holds:

\begin{itemize}
\item[\rm(1)] $|a_{i,i}| > r_i(A)$;

\item[\rm(2)] $|a_{i,i}| > c_i(A)$;

\item[\rm(3)] $|a_{i,i}| > r_i^{\alpha}(A)c_i^{1-\alpha}(A)$;

\item[\rm(4)] $|a_{i,i}| > \alpha r_i(A) + (1-\alpha)c_i(A)$;

\item[\rm(5)] $|a_{i,i}||a_{j,j}| > r_i(A)r_j(A)$;

\item[\rm(6)] $|a_{i,i}||a_{j,j}| > c_i(A)c_j(A)$;

\item[\rm(7)] $|a_{i,i}||a_{j,j}| > r_i(A)c_j(A)$;

\item[\rm(8)] $|a_{i,i}||a_{j,j}| >
[r_i(A)r_j(A)]^{\alpha}[c_i(A)c_j(A)]^{1-\alpha}$;

\item[\rm(9)] $|a_{i,i}||a_{j,j}| >
[r_i(A)c_j(A)]^{\alpha}[r_j(A)c_i(A)]^{1-\alpha}$;

\item[\rm(10)] $|a_{i,i}||a_{j,j}| >
\alpha r_i(A)r_j(A) + (1-\alpha)c_i(A)c_j(A)$;

\item[\rm(11)] $|a_{i,i}||a_{j,j}| >
\alpha r_i(A)c_j(A) + (1-\alpha)r_j(A)c_i(A)$;

\item[\rm(12)] $|a_{i,i}||a_{j,j}| >
 [\alpha r_i(A)+(1-\alpha)c_i(A)][\alpha r_j(A) + (1-\alpha)c_j(A)]$;

\item[\rm(13)]
$|a_{i,i}||a_{j,j}| >
 [\alpha r_i(A)+(1-\alpha)c_i(A)][\alpha c_j(A) + (1-\alpha)r_j(A)]$;

 \item[\rm(14)]
$|a_{i,i}||a_{j,j}| >
 [\alpha r_i(A)+(1-\alpha)c_j(A)][\alpha r_j(A) + (1-\alpha)c_i(A)]$;

\item[\rm(15)]
$|a_{i,i}||a_{j,j}| >
 [\alpha r_i(A)+(1-\alpha)r_j(A)][\alpha c_j(A) + (1-\alpha)c_i(A)]$;

 \item[\rm(16)]
$|a_{i,i}|^{\alpha}|a_{j,j}|^{1-\alpha} > r_i^{\alpha}(A)r_j^{1-\alpha}(A)$;

\item[\rm(17)] $|a_{i,i}|^{\alpha}|a_{j,j}|^{1-\alpha} > c_i^{\alpha}(A)c_j^{1-\alpha}(A)$;

\item[\rm(18)] $|a_{i,i}|^{\alpha}|a_{j,j}|^{1-\alpha} >
r_i^{\alpha}(A)c_j^{1-\alpha}(A)$;

\item[\rm(19)]
$|a_{i,i}|^{\alpha}|a_{j,j}|^{1-\alpha} >
\alpha r_i(A)+(1-\alpha)r_j(A)$;

\item[\rm(20)]
$|a_{i,i}|^{\alpha}|a_{j,j}|^{1-\alpha} >
\alpha c_i(A)+(1-\alpha)c_j(A)$;

\item[\rm(21)]
$|a_{i,i}|^{\alpha}|a_{j,j}|^{1-\alpha} >
\alpha r_i(A)+(1-\alpha)c_j(A)$;

\item[\rm(22)]
$|a_{i,i}|^{\alpha}|a_{j,j}|^{1-\alpha} > [r_i^{\beta}(A)c_i^{1-\beta}(A)]^\alpha [r_j^{\beta}(A)c_j^{1-\beta}(A)]^{1-\alpha}$;

\item[\rm(23)]
$|a_{i,i}|^{\alpha}|a_{j,j}|^{1-\alpha} > [r_i^{\beta}(A)c_i^{1-\beta}(A)]^\alpha [c_j^{\beta}(A)r_j^{1-\beta}(A)]^{1-\alpha}$;

\item[\rm(24)]
$|a_{i,i}|^{\alpha}|a_{j,j}|^{1-\alpha} > \beta r_i^{\alpha}(A)r_j^{1-\alpha}(A) + (1-\beta)c_i^{\alpha}(A)c_j^{1-\alpha}(A)$;

\item[\rm(25)]
$|a_{i,i}|^{\alpha}|a_{j,j}|^{1-\alpha} > \beta r_i^{\alpha}(A)c_j^{1-\alpha}(A) + (1-\beta)c_i^{\alpha}(A)r_j^{1-\alpha}(A)$;

\item[\rm(26)]
$|a_{i,i}|^\alpha|a_{j,j}|^{1-\alpha} >
 \alpha r_i^{\beta}(A)c_i^{1-\beta}(A) + (1-\alpha)r_j^{\beta}(A)c_j^{1-\beta}(A)$;

\item[\rm(27)]
$|a_{i,i}|^\alpha|a_{j,j}|^{1-\alpha} >
 \alpha r_i^{\beta}(A)c_i^{1-\beta}(A) + (1-\alpha)c_j^{\beta}(A)r_j^{1-\beta}(A)$;

\item[\rm(28)]
$|a_{i,i}|^{\alpha}|a_{j,j}|^{1-\alpha} > [\beta r_i(A)+(1-\beta)c_i(A)]^\alpha[\beta r_j(A) + (1-\beta)c_j(A)]^{1-\alpha}$;

\item[\rm(29)]
$|a_{i,i}|^{\alpha}|a_{j,j}|^{1-\alpha} > [\beta r_i(A)+(1-\beta)c_i(A)]^\alpha[\beta c_j(A) + (1-\beta)r_j(A)]^{1-\alpha}$;

\item[\rm(30)]
$|a_{i,i}|^\alpha|a_{j,j}|^{1-\alpha} >
[\alpha r_i(A)+(1-\alpha)c_j(A)]^{\beta}[\alpha c_i(A) + (1-\alpha)r_j(A)]^{1-\beta}$;

\item[\rm(31)]
$|a_{i,i}|^\alpha|a_{j,j}|^{1-\alpha} >
[\alpha r_i(A)+(1-\alpha)r_j(A)]^{\beta}[\alpha c_i(A) + (1-\alpha)c_j(A)]^{1-\beta}$.
\end{itemize}
\end{theorem}

Of course, we can also take $(g,h)$ in Theorems 4.1 and 4.4 as some other combinations, just
$g,h \in \{r,\tilde{r},c,\tilde{c},r^X,r^Y,c^X,c^Y,\tilde{r}^X,\tilde{r}^Y,\tilde{c}^X,\tilde{c}^Y\}$
for any $X,Y\in D^P_n$,
then we can export more sufficient conditions,
those are exactly like Theorems 4.5, 4.6 and 4.7.

The theorems above include some known results as follows.

The conditions (1) and (2) in Theorem 4.3 are consistent with definition
of generalized diagonally dominant matrices.

The condition (3) is taken as a sufficient condition for nonsingular matrix [52, Corollary 1.17].

The conditions (1)-(4), (8) and (12) in Theorem 4.7
as the sufficient conditions for $H$-matrices are given [37, p.189, 14].

When one of the conditions (3), (4) and (8)
in Theorem 4.7 holds, it is shown by [35, Satz I, p.210 and Satz II]
that $A$ is nonsingular. However, there it is no proof that $A \in GDD$.
While, in [17, (6.8), (6.9) and (6.10)] it is proved that $A$ is a nonsingular $M$-matrix whenever $A \in {\cal Z}$.
When (3) or (4) holds it is proved by [11, Theorem 13] that $A$ is an $H$-matrix.
Furthermore, (3) gives a positive answer to [15, Remark 6.5].

When the condition (5) in Theorem 4.7 holds, it is proved
that $A$ is nonsingular by [34], [5, p.22] and 46, Theorem V].
And it is shown that $A$ is an $H$-matrix by [12, Theorem 3].

The condition (7) in Theorem 4.7 is given in [43, Theorem 9].

\newpage\markboth{}{}
\section{Matrix eigenvalues inclusion regions}\label{se5}

The Ger\v{s}gorin Circle Theorem gives a matrix eigenvalues inclusion region.
In this section we propose some matrix eigenvalues inclusion regions
using the $G$-function pairs. The results include some known ones.

Denote
\begin{eqnarray*}
\Gamma_i(A):= \{z\in \mathbb{C}|\;|z-a_{i,i}| \le r_i(A)\}, \; i \in {\cal N},
\end{eqnarray*}
and
\begin{eqnarray*}
\Gamma(A):= \bigcup\limits_{i\in {\cal N}} \Gamma_i(A),
\end{eqnarray*}
where $\Gamma_i(A)$ is called the $i^{\mbox{th}}$-Ger\v{s}gorin disk
of $A$ having center $a_{i,i}$ and radius $r_i(A)$. And $\Gamma(A)$ is called the
Ger\v{s}gorin set.

The Ger\v{s}gorin Circle Theorem says that
for any $\lambda\in\sigma(A)$, there exists $i \in {\cal N}$
such that $\lambda \in \Gamma_i(A)$ and, therefore, $\sigma(A)\subseteq \Gamma(A)$.

To consider relations between the $G$-function pairs and
matrix eigenvalues inclusion regions, for $A\in \mathbb{C}^{n\times n}$
and $(g,h)\in \mathcal{G}^F$, we set
\begin{eqnarray*}
\left\{\begin{array}{rll}\Gamma_i(A,F,g,h) & := & \{z\in \mathbb{C}^n:[F(|D(zI-A)|e,|D(zI-A)|e)]_i\\
&& \le
[F(g(A),h(A))]_i\}, \; i = 1, \cdots, m,\\
\Gamma(A,F,g,h) & := & \bigcup\limits_{i=1}^m
\Gamma_i(A,F,g,h), \\
{\bf\Gamma}(A,F) & := & \bigcap\limits_{(g,h)\in\mathcal{G}^F}\Gamma(A,F,g,h).
\end{array} \right.
\end{eqnarray*}

Similar to the Ger\v{s}gorin circle theorem we prove the following eigenvalue inclusion theorem.

\begin{theorem} \label{thm5-1}
Suppose that $\mathcal{G}^F$ is a set of the $G$-function pair induced by $F$.
Then for any $A\in \mathbb{C}^{n\times n}$ and $(g,h)\in \mathcal{G}^F$, the following results hold:

\begin{itemize}
\item[\rm(a)] For any $\lambda\in\sigma(A)$, there exists $i\in\{1, \cdots, m\}$
such that   $\lambda \in \Gamma_i(A,F,g,h)$;

\item[\rm(b)] $\sigma(A)\subseteq {\bf\Gamma}(A,F)\subseteq \Gamma(A,F,g,h)$.
\end{itemize}
\end{theorem}

\begin{proof}
Assume that there exists $\lambda \in \sigma(A)$ such that for all $i\in\{1, \cdots, m\}$,
$\lambda \notin \Gamma_i(A,F,g,h)$,
i.e.,
\begin{eqnarray*}
[F(|D(\lambda I-A)|e,|D(\lambda I-A)|e)]_i > [F(g(A),h(A))]_i, \; i = 1, \cdots, m.
 \end{eqnarray*}

Let $\bar{A} = \lambda I - A$. Then $\bar{A}$ is singular. Furthermore, the off-diagonal entries of $\bar{A}$ and $-A$ are same.
Hence, for $i = 1, \cdots, m$, we have
\begin{eqnarray*}
[F(|D(\bar{A}|)e,|D(\bar{A})|e)]_i & = & [F(|D(\lambda I-A)|e,|D(\lambda I-A)|e)]_i \\
& > &  [F(g(A),h(A))]_i \\
& = &  [F(g(\bar{A}),h(\bar{A}))]_i.
 \end{eqnarray*}
Since $(g,h)\in \mathcal{G}^F$, it follows that $\bar{A}$ is nonsingular. This is a contradiction.

Now (a) has been proved.

By (a), (b) is obvious.
\end{proof}

Just as the equivalence between the Ger\v{s}gorin circle
theorem and the strictly diagonally dominant theorem, we can easily prove the following
inverse theorem.

\begin{theorem} \label{thm5-2}
Let $(g,h)\in \mathcal{F}_n\times \mathcal{F}_n$, and let
$F: \mathbb{R}_+^n\times \mathbb{R}_+^n\rightarrow \mathbb{R}_+^m$ for $m \ge 1$
is a monotonic function. Suppose that
for any $A\in \mathbb{C}^{n\times n}$ and for any $\lambda\in\sigma(A)$, there exists
$i\in\{1, \cdots, m\}$
such that $\lambda \in \Gamma_i(A,F,g,h)$. Then $(g,h)$ is a $G$-function pair.
\end{theorem}

Corresponding to the $G$-functions pairs given in Definition 3.3,
we define corresponding circles and ovals of Cassini.

\begin{definition} \label{den5-1}
For $A\in \mathbb{C}^{n\times n}$ and
$(g,h)\in \mathcal{F}_n\times \mathcal{F}_n$, $\alpha,\beta\in[0,1]$, $i,j \in {\cal N}$, $i \not= j$,
we define 27 kind of circles, ovals of Cassini and their sets as follows, where the symbol ``$\ast$'' denotes ``$g$", ``$g,h$", ``$g,\alpha$", ``$g,h,\alpha$" or ``$g,h,\alpha,\beta$", respectively.

\begin{itemize}
\item[$\bullet$] $\check\Gamma_i^{(k)}(A,\ast):=
\{z\in \mathbb{C}^n: |z - a_{i,i}| \le \rho_i^{(k)}(A,\ast)\}$, $k=1,2,3$, with

$\rho_i^{(1)}(A,g) = g_i(A)$,

$\rho_i^{(2)}(A,g,h,\alpha) = g_i^{\alpha}(A)h_i^{1-\alpha}(A)$,

$\rho_i^{(3)}(A,g,h,\alpha) = \alpha g_i(A) + (1-\alpha)h_i(A)$;

\item[$\circ$] $\check\Gamma_{i,j}^{(k)}(A,\ast):=
\{z\in \mathbb{C}^n: |z - a_{i,i}||z - a_{j,j}| \le \rho_{i,j}^{(k)}(A,\ast)\}$, $k=4,\cdots,13$, with

$\rho_{i,j}^{(4)}(A,g) = g_i(A)g_j(A)$,

$\rho_{i,j}^{(5)}(A,g,h) = g_i(A)h_j(A)$,

$\rho_{i,j}^{(6)}(A,g,h,\alpha) = [g_i(A)g_j(A)]^{\alpha}[h_i(A)h_j(A)]^{1-\alpha}$,

$\rho_{i,j}^{(7)}(A,g,h,\alpha) = [g_i(A)h_j(A)]^{\alpha}[g_j(A)h_i(A)]^{1-\alpha}$,

$\rho_{i,j}^{(8)}(A,g,h,\alpha) = \alpha g_i(A)g_j(A) + (1-\alpha)h_i(A)h_j(A)$,

$\rho_{i,j}^{(9)}(A,g,h,\alpha) = \alpha g_i(A)h_j(A) + (1-\alpha)g_j(A)h_i(A)$,

$\rho_{i,j}^{(10)}(A,g,h,\alpha) = [\alpha g_i(A)+(1-\alpha)h_i(A)][\alpha g_j(A) + (1-\alpha)h_j(A)]$,

$\rho_{i,j}^{(11)}(A,g,h,\alpha) = [\alpha g_i(A)+(1-\alpha)h_i(A)][\alpha h_j(A) + (1-\alpha)g_j(A)]$,

$\rho_{i,j}^{(12)}(A,g,h,\alpha) = [\alpha g_i(A)+(1-\alpha)h_j(A)][\alpha g_j(A) + (1-\alpha)h_i(A)]$,

$\rho_{i,j}^{(13)}(A,g,h,\alpha) = [\alpha g_i(A)+(1-\alpha)g_j(A)][\alpha h_j(A) + (1-\alpha)h_i(A)]$;

\item[$\circ$] $\check\Gamma_{i,j}^{(k)}(A,\ast):=
\{z\in \mathbb{C}^n: |z - a_{i,i}|^{\alpha}|z - a_{j,j}|^{1-\alpha} \le \rho_{i,j}^{(k)}(A,\ast)\}$, $k=14,\cdots,27$, with

$\rho_{i,j}^{(14)}(A,g,\alpha) =
 g_i^{\alpha}(A)g_j^{1-\alpha}(A)$,

$\rho_{i,j}^{(15)}(A,g,h,\alpha) =
g_i^{\alpha}(A)h_j^{1-\alpha}(A)$,

$\rho_{i,j}^{(16)}(A,g,\alpha) =
\alpha g_i(A)+(1-\alpha)g_j(A)$,

$\rho_{i,j}^{(17)}(A,g,h,\alpha) =
\alpha g_i(A)+(1-\alpha)h_j(A)$,

$\rho_{i,j}^{(18)}(A,g,h,\alpha,\beta) =
[g_i^{\beta}(A)h_i^{1-\beta}(A)]^\alpha [g_j^{\beta}(A)h_j^{1-\beta}(A)]^{1-\alpha}$,

$\rho_{i,j}^{(19)}(A,g,h,\alpha,\beta) =
[g_i^{\beta}(A)h_i^{1-\beta}(A)]^\alpha
[h_j^{\beta}(A) g_j^{1-\beta}(A)]^{1-\alpha}$,

$\rho_{i,j}^{(20)}(A,g,h,\alpha,\beta) =
\beta g_i^\alpha(A)g_j^{1-\alpha}(A) + (1-\beta)h_i^\alpha(A) h_j^{1-\alpha}(A)$,

$\rho_{i,j}^{(21)}(A,g,h,\alpha,\beta) =
\beta g_i^\alpha(A) h_j^{1-\alpha}(A) + (1-\beta)h_i^\alpha(A) g_j^{1-\alpha}(A)$,

$\rho_{i,j}^{(22)}(A,g,h,\alpha,\beta) =
\alpha g_i^{\beta}(A)h_i^{1-\beta}(A) + (1-\alpha)g_j^{\beta}(A)h_j^{1-\beta}(A)$,

$\rho_{i,j}^{(23)}(A,g,h,\alpha,\beta) =
\alpha g_i^{\beta}(A)h_i^{1-\beta}(A) + (1-\alpha)h_j^{\beta}(A)g_j^{1-\beta}(A)$,

$\rho_{i,j}^{(24)}(A,g,h,\alpha,\beta) =
[\beta g_i(A)+(1-\beta)h_i(A)]^\alpha[\beta g_j(A) + (1-\beta)h_j(A)]^{1-\alpha}$,

$\rho_{i,j}^{(25)}(A,g,h,\alpha,\beta) =
[\beta g_i(A)+(1-\beta)h_i(A)]^\alpha[\beta h_j(A) + (1-\beta)g_j(A)]^{1-\alpha}$,

$\rho_{i,j}^{(26)}(A,g,h,\alpha,\beta) =
[\alpha g_i(A)+(1-\alpha)g_j(A)]^{\beta}[\alpha h_i(A) + (1-\alpha)h_j(A)]^{1-\beta}$,

$\rho_{i,j}^{(27)}(A,g,h,\alpha,\beta) =
[\alpha g_i(A)+(1-\alpha)h_j(A)]^{\beta}[\alpha h_i(A) + (1-\alpha)g_j(A)]^{1-\beta}$
\end{itemize}
and
\begin{itemize}

\item[$\bullet$]
${\check\Gamma}^{(k)}(A,\ast):= \bigcup\limits_{i=1}^n\check\Gamma_i^{(k)}(A,\ast)$, $k=1,2,3$,

${\check\Gamma}^{(k)}(A,\ast):= \bigcup\limits_{i,j=1,i\not=j}^n\check\Gamma_{i,j}^{(k)}(A,\ast)$,
$k=4, \cdots, 27$;

\item[$\bullet$] ${\bf\check\Gamma}^{(k)}(A,\mathcal{G}_n):= \bigcap\limits_{g \in \mathcal{G}_n}{\check\Gamma}^{(k)}(A,g)$, $k=1,4$,

${\bf\check\Gamma}^{(5)}(A,\mathcal{G}_n):= \bigcap\limits_{g,h \in \mathcal{G}_n}{\check\Gamma}^{(k)}(A,g,h)$,

${\bf\check\Gamma}^{(k)}(A,\mathcal{G}_n):= \bigcap\limits_{g\in \mathcal{G}_n; \alpha\in[0,1]} {\check\Gamma}^{(k)}(A,g,\alpha)$,
$k=14,16$,

${\bf\check\Gamma}^{(k)}(A,\mathcal{G}_n):= \bigcap\limits_{g,h \in \mathcal{G}_n; \alpha\in[0,1]} {\check\Gamma}^{(k)}(A,g,h,\alpha)$,
$k=2,3,6,\cdots,13$,15,17,

${\bf\check\Gamma}^{(k)}(A,\mathcal{G}_n):= \bigcap\limits_{g,h \in \mathcal{G}_n; \alpha,\beta\in[0,1]}{\check\Gamma}^{(k)}(A,g,h,\alpha,\beta)$, \;
$k=18,\cdots,27$.
\end{itemize}
\end{definition}

Obviously, $k=1,\cdots,27$, ${\check\Gamma}^{(k)}(A,\ast)$ are easy to determine numerically. But ${\bf\check\Gamma}^{(k)}(A,\mathcal{G}_n)$ are not the case,
for even relatively low values of $n$. Because of this, their values are more for theoretical purposes.
There will be some similar situations below.

The following lemma will give some relationships respectively among ${\check\Gamma}^{(k)}(A,\\\ast)$ and
${\bf\check\Gamma}^{(k)}(A,\mathcal{G}_n)$, $k=1,\cdots,27$.

\begin{lemma} \label{lem5-1}
For $A\in \mathbb{C}^{n\times n}$ and
$(g,h)\in \mathcal{G}_n\times \mathcal{G}_n$, $\alpha,\beta\in[0,1]$,
let ${\check\Gamma}^{(k)}(A,\ast)$ and
${\bf\check\Gamma}^{(k)}(A,\mathcal{G}_n)$ be defined by Definition 5.1, $k=1,\cdots,27$. Then

\begin{itemize}
\item[\rm(1)]
${\check\Gamma}^{(1)}(A,g) = {\check\Gamma}^{(\xi)}(A,g,g,\alpha) = {\check\Gamma}^{(\varsigma)}(A,g,1) = {\check\Gamma}^{(\mu)}(A,g,h,1)  \\
= {\check\Gamma}^{(\nu)}(A,g,h,1,1) = {\check\Gamma}^{(\zeta)}(A,g,h,0,1) = {\check\Gamma}^{(\eta)}(A,g,h,0,0) \\
\supseteq [{\check\Gamma}^{(4)}(A,g)\cup{\check\Gamma}^{(14)}(A,g,\alpha)]$,

$\xi=2,3$, $\varsigma=14,16$, $\mu=2,3,15,17$, $\nu=18,\cdots,27$, \\ $\zeta=18,20,22,24,26$, $\eta=19,21,23,25,27$;

\item[\rm(2)]
${\check\Gamma}^{(2)}(A,g,h,\alpha) = {\check\Gamma}^{(k)}(A,g,h,1,\alpha)  \supseteq
[{\check\Gamma}^{(6)}(A,g,h,\alpha)\cup{\check\Gamma}^{(18)}(A,g,h,\alpha,\alpha)]$,
$k=18,19,22,23,26,27$;

\item[\rm(3)]
${\check\Gamma}^{(3)}(A,g,h,\alpha) = {\check\Gamma}^{(\mu)}(A,g,h,1,\alpha)   \supseteq
[{\check\Gamma}^{(\nu)}(A,g,h,\alpha)\cup{\check\Gamma}^{(24)}(A,g,h,\alpha,\alpha)]$,
$\mu=20,21,24,25$, $\nu = 2,10$;

\item[\rm(4)]
${\check\Gamma}^{(4)}(A,g) = {\check\Gamma}^{(14)}(A,g,\frac{1}{2}) = {\check\Gamma}^{(5)}(A,g,g) = {\check\Gamma}^{(\zeta)}(A,g,g,\alpha)
\\ = {\check\Gamma}^{(\mu)}(A,g,h,1)
 = {\check\Gamma}^{(\nu)}(A,g,h,\frac{1}{2},1)$,

$\zeta=6,\cdots,11$, $\mu=6,8,10,12$, $\nu=18,20,24$;

\item[\rm(5)]
${\check\Gamma}^{(5)}(A,g,h) = {\check\Gamma}^{(15)}(A,g,h,\frac{1}{2}) = {\check\Gamma}^{(\mu)}(A,g,h,1) = {\check\Gamma}^{(\nu)}(A,g,h,\frac{1}{2},1)$,

$\mu=7,9,11,13$, $\nu=19,21,25$;

\item[\rm(6)]
${\check\Gamma}^{(6)}(A,g,h,\alpha) = {\check\Gamma}^{(18)}(A,g,h,\frac{1}{2},\alpha) \subseteq {\check\Gamma}^{(k)}(A,g,h,\alpha)$, $k=8,10,12$;

\item[\rm(7)]
${\check\Gamma}^{(7)}(A,g,h,\alpha) = {\check\Gamma}^{(19)}(A,g,h,\frac{1}{2},\alpha) \subseteq {\check\Gamma}^{(k)}(A,g,h,\alpha)$, $k=9,11,13$;

\item[\rm(8)]
${\check\Gamma}^{(10)}(A,g,h,\alpha) = {\check\Gamma}^{(24)}(A,g,h,\frac{1}{2},\alpha)$;

\item[\rm(9)]
${\check\Gamma}^{(11)}(A,g,h,\alpha) = {\check\Gamma}^{(25)}(A,g,h,\frac{1}{2},\alpha)$;

\item[\rm(10)]
${\check\Gamma}^{(14)}(A,g,\alpha) = {\check\Gamma}^{(15)}(A,g,g,\alpha) = {\check\Gamma}^{(\mu)}(A,g,g,\alpha,\beta) = {\check\Gamma}^{(\nu)}(A,g,h,\alpha,1)\\
\subseteq {\check\Gamma}^{(16)}(A,g,\alpha) = {\check\Gamma}^{(17)}(A,g,g,\alpha)  = {\check\Gamma}^{(\zeta)}(A,g,h,\alpha,1)$,\\
$\mu=18,19,20,21,24,25$, $\nu=18,20,24$, $\zeta=22,26$;

\item[\rm(11)]
${\check\Gamma}^{(15)}(A,g,h,\alpha) = {\check\Gamma}^{(\mu)}(A,g,h,\alpha,1) \subseteq {\check\Gamma}^{(17)}(A,g,h,\alpha) \\ = {\check\Gamma}^{(\nu)}(A,g,h,\alpha,1)$,
$\mu=19,21,25$, $\nu=23,27$;

\item[\rm(12)]
${\check\Gamma}^{(18)}(A,g,h,\alpha,\beta) \subseteq {\check\Gamma}^{(k)}(A,g,h,\alpha,\beta)$, $k=20,22,24,26$;

\item[\rm(13)]
${\check\Gamma}^{(19)}(A,g,h,\alpha,\beta) \subseteq {\check\Gamma}^{(k)}(A,g,h,\alpha,\beta)$, $k=21,23,25,27$.

\end{itemize}

Furthermore, it holds that

\begin{itemize}
\item[\rm(a)]
${\bf\check\Gamma}^{(1)}(A,\mathcal{G}_n) = {\bf\check\Gamma}^{(k)}(A,\mathcal{G}_n)$,
$k=2,3,4,6,8,10,12$;

\item[\rm(b)]
${\bf\check\Gamma}^{(5)}(A,\mathcal{G}_n) = {\bf\check\Gamma}^{(k)}(A,\mathcal{G}_n)$, $k=7,9,11,13$;

\item[\rm(c)]
${\bf\check\Gamma}^{(14)}(A,\mathcal{G}_n) = {\bf\check\Gamma}^{(k)}(A,\mathcal{G}_n)$,
$k=18,20,24$;

\item[\rm(d)]
${\bf\check\Gamma}^{(15)}(A,\mathcal{G}_n) = {\bf\check\Gamma}^{(k)}(A,\mathcal{G}_n)$,
$k=19,21,25$;

\item[\rm(e)]
${\bf\check\Gamma}^{(16)}(A,\mathcal{G}_n) = {\bf\check\Gamma}^{(22)}(A,\mathcal{G}_n)$;

\item[\rm(f)]
${\bf\check\Gamma}^{(17)}(A,\mathcal{G}_n)= {\bf\check\Gamma}^{(23)}(A,\mathcal{G}_n)$;

\item[\rm(g)]
${\bf\check\Gamma}^{(15)}(A,\mathcal{G}_n) \subseteq {\bf\check\Gamma}^{(5)}(A,\mathcal{G}_n)\subseteq {\bf\check\Gamma}^{(1)}(A,\mathcal{G}_n)$;

\item[\rm(h)]
${\bf\check\Gamma}^{(15)}(A,\mathcal{G}_n) \subseteq {\bf\check\Gamma}^{(14)}(A,\mathcal{G}_n)\subseteq {\bf\check\Gamma}^{(26)}(A,\mathcal{G}_n) \subseteq {\bf\check\Gamma}^{(16)}(A,\mathcal{G}_n)   \subseteq {\bf\check\Gamma}^{(1)}(A,\mathcal{G}_n)$;

\item[\rm(i)]
${\bf\check\Gamma}^{(15)}(A,\mathcal{G}_n) \subseteq {\bf\check\Gamma}^{(27)}(A,\mathcal{G}_n)\subseteq {\bf\check\Gamma}^{(17)}(A,\mathcal{G}_n) \subseteq {\bf\check\Gamma}^{(16)}(A,\mathcal{G}_n)$.
\end{itemize}
\end{lemma}

\begin{proof}
By  Lemma 2.5, we can derive (1)-(13),
where ${\check\Gamma}^{(1)}(A,g)
\supseteq {\check\Gamma}^{(4)}(A,g)$
follows by [52, Lemma 5.22] and it can be proved similarly that
${\check\Gamma}^{(1)}(A,g)
\supseteq {\check\Gamma}^{(14)}(A,g,\alpha)$,
${\check\Gamma}^{(2)}(A,g,h,\alpha) \supseteq [{\check\Gamma}^{(6)}(A,g,h,\alpha)\cup
{\check\Gamma}^{(18)}(A,g,h,\alpha,\alpha)]$ and
${\check\Gamma}^{(3)}(A,\\g,h,\alpha) \supseteq [{\check\Gamma}^{(10)}(A,g,h,\alpha)\cup
{\check\Gamma}^{(24)}(A,g,h,\alpha,\alpha)]$.

Now we prove (a)-(i).

From (1) and (3) we have ${\bf\check\Gamma}^{(2)}(A,\mathcal{G}_n) \subseteq
{\bf\check\Gamma}^{(3)}(A,\mathcal{G}_n) \subseteq {\bf\check\Gamma}^{(1)}(A,\mathcal{G}_n)$.
While, since $f:= g^\alpha h^{1-\alpha} \in \mathcal{G}_n$ it follows that
${\bf\check\Gamma}^{(1)}(A,\mathcal{G}_n) \subseteq
{\bf\check\Gamma}^{(2)}(A,\mathcal{G}_n)$, so that ${\bf\check\Gamma}^{(1)}(A,\mathcal{G}_n) = {\bf\check\Gamma}^{(2)}(A,\mathcal{G}_n) = {\bf\check\Gamma}^{(3)}(A,\mathcal{G}_n)$ holds.
From (4) and (6) we have ${\bf\check\Gamma}^{(6)}(A,\mathcal{G}_n) \subseteq
{\bf\check\Gamma}^{(k)}(A,\mathcal{G}_n)
\subseteq {\bf\check\Gamma}^{(4)}(A,\mathcal{G}_n)$ for $k=8,10$.
While, since $f= g^\alpha h^{1-\alpha} \in \mathcal{G}_n$, it follows that
${\bf\check\Gamma}^{(4)}(A,\mathcal{G}_n)
\subseteq
{\bf\check\Gamma}^{(6)}(A,\mathcal{G}_n)$, so that ${\bf\check\Gamma}^{(4)}(A,\mathcal{G}_n) =
{\bf\check\Gamma}^{(k)}(A,\mathcal{G}_n)$ for $k=6,8,10$.
Now, for $k =1,4$, by Lemmas 2.3 and 2.3, we derive directly
\begin{eqnarray*}
{\bf\check\Gamma}^{(k)}(A,\mathcal{G}_n)
= \bigcap\limits_{g \in \mathcal{G}_n}{\check\Gamma}^{(k)}(A,g) = \bigcap\limits_{X \in D^P_n}{\check\Gamma}^{(k)}(A,\tilde{r}^X).
\end{eqnarray*}
By [51, Theorem 5] or [52, Theorem 4.15]
it can be proved that ${\bf\check\Gamma}^{(1)}(A,\mathcal{G}_n) =
{\bf\check\Gamma}^{(4)}(A,\mathcal{G}_n)$.
By (6) it gets that ${\bf\check\Gamma}^{(12)}(A,\mathcal{G}_n) \supseteq {\bf\check\Gamma}^{(6)}(A,\mathcal{G}_n)$.
While, by (4) we have
${\bf\check\Gamma}^{(12)}(A,\mathcal{G}_n) \subseteq {\bf\check\Gamma}^{(4)}(A,\mathcal{G}_n)$. Hence
${\bf\check\Gamma}^{(12)}(A,\mathcal{G}_n) = {\bf\check\Gamma}^{(4)}(A,\mathcal{G}_n)$.
This shows (a).

From (5) and (7) we can prove ${\bf\check\Gamma}^{(7)}(A,\mathcal{G}_n) \subseteq {\bf\check\Gamma}^{(k)}(A,\mathcal{G}_n) \subseteq {\bf\check\Gamma}^{(5)}(A,\mathcal{G}_n)$
for $k=9,11,13$. While, since $f=g^\alpha h^{1-\alpha}\in \mathcal{G}_n$, $\tilde{f}:=h^\alpha g^{1-\alpha} \in \mathcal{G}_n$, then
${\bf\check\Gamma}^{(7)}(A,\mathcal{G}_n) \supseteq {\bf\check\Gamma}^{(5)}(A,\mathcal{G}_n)$. Hence (b) holds.

By (10) and (12) it gets that ${\bf\check\Gamma}^{(18)}(A,\mathcal{G}_n) \subseteq {\bf\check\Gamma}^{(k)}(A,\mathcal{G}_n)
\subseteq {\bf\check\Gamma}^{(14)}(A,\mathcal{G}_n)$, $k=20,24$. While, since $\hat{f}:=g^\beta h^{1-\beta} \in \mathcal{G}_n$, then
${\bf\check\Gamma}^{(14)}(A,\mathcal{G}_n) \subseteq {\bf\check\Gamma}^{(18)}(A,\mathcal{G}_n)$, so that ${\bf\check\Gamma}^{(14)}(A,\mathcal{G}_n) =
{\bf\check\Gamma}^{(k)}(A,\mathcal{G}_n)$, $k=18,20,24$.
(c) has been proved.

Similarly, by (11), (13) and $\hat{f}=g^\beta h^{1-\beta}\in \mathcal{G}_n$, $\check{f}:=h^\beta g^{1-\beta} \in \mathcal{G}_n$, we can prove (d).

It gets that
${\bf\check\Gamma}^{(16)}(A,\mathcal{G}_n)
\supseteq {\bf\check\Gamma}^{(22)}(A,\mathcal{G}_n)$ by (10) and ${\bf\check\Gamma}^{(16)}(A,\mathcal{G}_n)
\subseteq {\bf\check\Gamma}^{(22)}(A,\\\mathcal{G}_n)$ by $\hat{f}, \check{f} \in \mathcal{G}_n$, so that (e) holds.

The proof of (f) is completely similar.

We can prove (g) from (4) and (5), (h) from (1), (10) and (12), (i) from (10), (11) and (13).
\end{proof}

Although we have proved that some ${\bf\check\Gamma}^{(k)}(A,\mathcal{G}_n)$ are same sets, it does not mean that the corresponding ${\check\Gamma}^{(k)}(A,\ast)$ are the same one. There are some similar situations below.

Based on Theorems 4.1, 5.1 and Lemma 5.1,
the following eigenvalue inclusion theorem is directly.

\begin{theorem} \label{thm5-3}
For $A\in \mathbb{C}^{n\times n}$ and
$(g,h)\in \mathcal{G}_n\times \mathcal{G}_n$, $\alpha,\beta\in[0,1]$,
let ${\check\Gamma}^{(k)}(A,\ast)$ and
${\bf\check\Gamma}^{(k)}(A,\mathcal{G}_n)$ be defined by Definition 5.1, $k=1,\cdots,27$. Then the following results hold:

\begin{itemize}
\item[\rm(a)]
For any $\lambda \in \sigma(A)$, there exists $i_k \in {\cal N}$ such that

$\lambda \in
\check\Gamma_{i_k}^{(k)}(A,\ast)$ for $k=1,2,3$,

and there exist $i_k, j_k \in {\cal N}$ with $i_k\not=j_k$ such that

$\lambda \in
\check\Gamma_{i_k,j_k}^{(k)}(A,\ast)$ for $k=4,\cdots,27$;

\item[\rm(b)]
$\sigma(A) \subseteq {\bf\check\Gamma}^{(k)}(A,\mathcal{G}_n) \subseteq {\check\Gamma}^{(k)}(A,\ast)$ for
$k=1,\cdots,27$;

\item[\rm(c)]
$\sigma(A) \subseteq \bigcap\limits_{1\le k \le 27}{\bf\check\Gamma}^{(k)}(A,\mathcal{G}_n) = {\bf\check\Gamma}^{(15)}(A,\mathcal{G}_n)$.
\end{itemize}
\end{theorem}

From this theorem and Lemma 5.1 it is easy to see that ${\bf\check\Gamma}^{(15)}(A,\mathcal{G}_n) = {\bf\check\Gamma}^{(\mu)}(A,\mathcal{G}_n)$ with $\mu=19,21,25$
is the minimal eigenvalue inclusion set among ${\bf\check\Gamma}^{(k)}(A,\mathcal{G}_n)$ and
${\check\Gamma}^{(k)}(A,\ast)$, $k=1,\cdots,27$.

When the $G$-function pair $(g,h)$ in Definition 5.1 equals $(\tilde{r}^X, \tilde{r}^X)$, $(\tilde{c}^Y, \tilde{c}^Y)$, $(\tilde{r}^X, \tilde{c}^Y)$, $(r^X, r^X)$, $(c^Y, c^Y)$ or $(r^X, c^Y)$, we can derive some special circles and ovals of Cassini.

\begin{definition} \label{den5-2}
For $A\in \mathbb{C}^{n\times n}$ and
$X,Y \in D^P_n$, $\alpha,\beta\in[0,1]$, $i,j \in {\cal N}$, $i \not= j$,
we define 31 kind of circles, ovals of Cassini and their sets as follows, where the symbol ``$\star$"
denotes ``$\tilde{r}^X$", ``$\tilde{c}^Y$", ``$\tilde{r}^X,\alpha$", ``$\tilde{c}^Y,\alpha$", ``$\tilde{r}^X,\tilde{c}^Y$", ``$\tilde{r}^X,\tilde{c}^Y,\alpha$", ``$\tilde{r}^X,\tilde{c}^Y,\alpha,\beta$", respectively.

\begin{itemize}
\item[$\bullet$] $\hat{\Gamma}_i^{(k)}(A,\star):=
\{z\in \mathbb{C}^n: |z - a_{i,i}| \le \rho_i^{(k)}(A,\star)\}$, $k=1,2,3,4$, with

$\rho_i^{(1)}(A,\tilde{r}^X) = \tilde{r}_i^X(A)$,

$\rho_i^{(2)}(A,\tilde{c}^Y) = \tilde{c}_i^Y(A)$,

$\rho_i^{(3)}(A,\tilde{r}^X,\tilde{c}^Y,\alpha) = [\tilde{r}_i^X(A)]^{\alpha}[\tilde{c}_i^Y(A)]^{1-\alpha}$,

$\rho_i^{(4)}(A,\tilde{r}^X,\tilde{c}^Y,\alpha) = \alpha\tilde{r}_i^X(A) + (1-\alpha)\tilde{c}_i^Y(A)$;

\item[$\circ$] $\hat{\Gamma}_{i,j}^{(k)}(A,\star):=
\{z\in \mathbb{C}^n: |z - a_{i,i}| |z - a_{j,j}| \le \rho_{i,j}^{(k)}(A,\star)\}$,
$k=5,\cdots,15$, with

$\rho_{i,j}^{(5)}(A,\tilde{r}^X) =
\tilde{r}_i^X(A)\tilde{r}_j^X(A)$,

$\rho_{i,j}^{(6)}(A,\tilde{c}^Y) =
\tilde{c}_i^Y(A)\tilde{c}_j^Y(A)$,

$\rho_{i,j}^{(7)}(A,\tilde{r}^X,\tilde{c}^Y) =
\tilde{r}_i^X(A)\tilde{c}_j^Y(A)$,

$\rho_{i,j}^{(8)}(A,\tilde{r}^X,\tilde{c}^Y,\alpha) =
[\tilde{r}_i^X(A)\tilde{r}_j^X(A)]^{\alpha}[\tilde{c}_i^Y(A)\tilde{c}_j^Y(A)]^{1-\alpha}$,

$\rho_{i,j}^{(9)}(A,\tilde{r}^X,\tilde{c}^Y,\alpha) =
[\tilde{r}_i^X(A)\tilde{c}_j^Y(A)]^{\alpha}[\tilde{r}_j^X(A)\tilde{c}_i^Y(A)]^{1-\alpha}$,

$\rho_{i,j}^{(10)}(A,\tilde{r}^X,\tilde{c}^Y,\alpha) =
\alpha\tilde{r}_i^X(A)\tilde{r}_j^X(A) + (1-\alpha)\tilde{c}_i^Y(A)\tilde{c}_j^Y(A)$,

$\rho_{i,j}^{(11)}(A,\tilde{r}^X,\tilde{c}^Y,\alpha) =
\alpha\tilde{r}_i^X(A)\tilde{c}_j^Y(A) + (1-\alpha)\tilde{r}_j^X(A)\tilde{c}_i^Y(A)$,

$\rho_{i,j}^{(12)}(A,\tilde{r}^X,\tilde{c}^Y,\alpha) =
[\alpha\tilde{r}_i^X(A)+(1-\alpha)\tilde{c}_i^Y(A)][\alpha\tilde{r}_j^X(A) + (1-\alpha)\tilde{c}_j^Y(A)]$,

$\rho_{i,j}^{(13)}(A,\tilde{r}^X,\tilde{c}^Y,\alpha) =
[\alpha\tilde{r}_i^X(A)+(1-\alpha)\tilde{c}_i^Y(A)][\alpha\tilde{c}_j^Y(A) + (1-\alpha)\tilde{r}_j^X(A)]$,

$\rho_{i,j}^{(14)}(A,\tilde{r}^X,\tilde{c}^Y,\alpha) =
[\alpha\tilde{r}_i^X(A)+(1-\alpha)\tilde{c}_j^Y(A)][\alpha\tilde{r}_j^X(A) + (1-\alpha)\tilde{c}_i^Y(A)]$,

$\rho_{i,j}^{(15)}(A,\tilde{r}^X,\tilde{c}^Y,\alpha) =
[\alpha\tilde{r}_i^X(A)+(1-\alpha)\tilde{r}_j^X(A)][\alpha\tilde{c}_j^Y(A) + (1-\alpha)\tilde{c}_i^Y(A)]$;

\item[$\circ$] $\hat{\Gamma}_{i,j}^{(k)}(A,\star):=
\{z\in \mathbb{C}^n: |z - a_{i,i}|^{\alpha}|z - a_{j,j}|^{1-\alpha} \le \rho_{i,j}^{(k)}(A,\star)\}$,
$k=16,\cdots,31$, with

$\rho_{i,j}^{(16)}(A,\tilde{r}^X,\alpha) =
[\tilde{r}_i^X(A)]^{\alpha}[\tilde{r}_j^X(A)]^{1-\alpha}$,

$\rho_{i,j}^{(17)}(A,\tilde{c}^Y,\alpha) =
[\tilde{c}_i^Y(A)]^{\alpha}[\tilde{c}_j^Y(A)]^{1-\alpha}$;

$\rho_{i,j}^{(18)}(A,\tilde{r}^X,\tilde{c}^Y,\alpha) =
[\tilde{r}_i^X(A)]^{\alpha}[\tilde{c}_j^Y(A)]^{1-\alpha}$,

$\rho_{i,j}^{(19)}(A,\tilde{r}^X,\alpha) =
\alpha \tilde{r}_i^X(A)+(1-\alpha)\tilde{r}_j^X(A)$,

$\rho_{i,j}^{(20)}(A,\tilde{c}^Y,\alpha) =
\alpha \tilde{c}_i^Y(A)+(1-\alpha)\tilde{c}_j^Y(A)$,

$\rho_{i,j}^{(21)}(A,\tilde{r}^X,\tilde{c}^Y,\alpha) =
\alpha \tilde{r}_i^X(A)+(1-\alpha)\tilde{c}_j^Y(A)$,

{\small $\rho_{i,j}^{(22)}(A,\tilde{r}^X,\tilde{c}^Y,\alpha,\beta) =
([\tilde{r}_i^X(A)]^{\beta}[\tilde{c}_i^Y(A)]^{1-\beta})^\alpha ([\tilde{r}_j^X(A)]^{\beta}[\tilde{c}_j^Y(A)]^{1-\beta})^{1-\alpha}$,

$\rho_{i,j}^{(23)}(A,\tilde{r}^X,\tilde{c}^Y,\alpha,\beta) =
([\tilde{r}_i^X(A)]^{\beta}[\tilde{c}_i^Y(A)]^{1-\beta})^\alpha
([\tilde{c}_j^Y(A)]^{\beta} [\tilde{r}_j^X(A)]^{1-\beta})^{1-\alpha}$,

$\rho_{i,j}^{(24)}(A,\tilde{r}^X,\tilde{c}^Y,\alpha,\beta) =
\beta [\tilde{r}_i^X(A)]^\alpha[\tilde{r}_j^X(A)]^{1-\alpha} + (1-\beta)[\tilde{c}_i^Y(A)]^\alpha [\tilde{c}_j^Y(A)]^{1-\alpha}$,

$\rho_{i,j}^{(25)}(A,\tilde{r}^X,\tilde{c}^Y,\alpha,\beta) =
\beta [\tilde{r}_i^X(A)]^\alpha [\tilde{c}_j^Y(A)]^{1-\alpha} + (1-\beta)[\tilde{c}_i^Y(A)]^\alpha [\tilde{r}_j^X(A)]^{1-\alpha}$,

$\rho_{i,j}^{(26)}(A,\tilde{r}^X,\tilde{c}^Y,\alpha,\beta) =
\alpha [\tilde{r}_i^X(A)]^{\beta}[\tilde{c}_i^Y(A)]^{1-\beta} + (1-\alpha)[\tilde{r}_j^X(A)]^{\beta}[\tilde{c}_j^Y(A)]^{1-\beta}$,

$\rho_{i,j}^{(27)}(A,\tilde{r}^X,\tilde{c}^Y,\alpha,\beta) =
\alpha [\tilde{r}_i^X(A)]^{\beta}[\tilde{c}_i^Y(A)]^{1-\beta} + (1-\alpha)[\tilde{c}_j^Y(A)]^{\beta}[\tilde{r}_j^X(A)]^{1-\beta}$,

$\rho_{i,j}^{(28)}(A,\tilde{r}^X,\tilde{c}^Y,\alpha,\beta) =
[\beta \tilde{r}_i^X(A)+(1-\beta)\tilde{c}_i^Y(A)]^\alpha[\beta \tilde{r}_j^X(A) + (1-\beta)\tilde{c}_j^Y(A)]^{1-\alpha}$,

$\rho_{i,j}^{(29)}(A,\tilde{r}^X,\tilde{c}^Y,\alpha,\beta) =
[\beta \tilde{r}_i^X(A)+(1-\beta)\tilde{c}_i^Y(A)]^\alpha[\beta \tilde{c}_j^Y(A) + (1-\beta)\tilde{r}_j^X(A)]^{1-\alpha}$,

$\rho_{i,j}^{(30)}(A,\tilde{r}^X,\tilde{c}^Y,\alpha,\beta) =
[\alpha \tilde{r}_i^X(A)+(1-\alpha)\tilde{r}_j^X(A)]^{\beta}[\alpha \tilde{c}_i^Y(A) + (1-\alpha)\tilde{c}_j^Y(A)]^{1-\beta}$,

$\rho_{i,j}^{(31)}(A,\tilde{r}^X,\tilde{c}^Y,\alpha,\beta) =
[\alpha \tilde{r}_i^X(A)+(1-\alpha)\tilde{c}_j^Y(A)]^{\beta}[\alpha \tilde{c}_i^Y(A) + (1-\alpha)\tilde{r}_j^X(A)]^{1-\beta}$}
\end{itemize}
and
  \begin{itemize}

 \item[$\bullet$] ${\hat\Gamma}^{(k)}(A,\star):= \bigcup\limits_{i=1}^n\hat{\Gamma}_i^{(k)}(A,\star)$, $k=1,2,3,4$,

${\hat\Gamma}^{(k)}(A,\star):= \bigcup\limits_{i,j=1,i\not=j}^n\hat{\Gamma}_{i,j}^{(k)}(A,\star)$,
$k=5, \cdots, 31$;

\item[$\bullet$]
${\bf\hat\Gamma}^{(k)}(A,\tilde{r},\tilde{c}):= \bigcap\limits_{X \in D^P_n}{\hat\Gamma}^{(k)}(A,\tilde{r}^X)$, $k=1,5$,

${\bf\hat\Gamma}^{(k)}(A,\tilde{r},\tilde{c}):= \bigcap\limits_{Y \in D^P_n}{\hat\Gamma}^{(k)}(A,\tilde{c}^Y)$, $k=2,6$,

${\bf\hat\Gamma}^{(k)}(A,\tilde{r},\tilde{c}):= \bigcap\limits_{X,Y \in D^P_n;\alpha\in[0,1]}{\hat\Gamma}^{(k)}(A,\tilde{r}^X,\tilde{c}^Y,\alpha)$,

\hskip 2cm $k=3,4,8,\cdots,15,18,21$,

${\bf\hat\Gamma}^{(7)}(A,\tilde{r},\tilde{c}):= \bigcap\limits_{X,Y \in D^P_n}{\hat\Gamma}^{(7)}(A,\tilde{r}^X,\tilde{c}^Y)$,

${\bf\hat\Gamma}^{(k)}(A,\tilde{r},\tilde{c}):= \bigcap\limits_{X \in D^P_n;\alpha\in[0,1]}{\hat\Gamma}^{(k)}(A,\tilde{r}^X,\alpha)$,
$k=16,19$,

${\bf\hat\Gamma}^{(k)}(A,\tilde{r},\tilde{c}):= \bigcap\limits_{Y \in D^P_n;\alpha\in[0,1]}{\hat\Gamma}^{(k)}(A,\tilde{c}^Y,\alpha)$,
$k=17,20$,

${\bf\hat\Gamma}^{(k)}(A,\tilde{r},\tilde{c}):=
\bigcap\limits_{X,Y \in D^P_n;\alpha,\beta\in[0,1]}{\hat\Gamma}^{(k)}(A,\tilde{r}^X,\tilde{c}^Y,\alpha,\beta)$,
$k=22,\cdots,31$.
 \end{itemize}
\end{definition}

For $k=1$, although it is independent of $\tilde{c}$, we still use sign
${\bf\hat\Gamma}^{(1)}(A,\tilde{r},\tilde{c})$ for the sake of simplicity and unity.
There are several similar cases behind.

Obviously, for $k=1,\cdots,31$, ${\hat\Gamma}^{(k)}(A,\star)$ are easy to determine numerically, but ${\bf\hat\Gamma}^{(k)}(A,\tilde{r},\tilde{c})$ are not the case. The same is true for several cases below.

Similar to Lemma 5.1, we have the relationships among  ${\hat\Gamma}^{(k)}(A,\star)$ and
${\bf\hat\Gamma}^{(k)}(A,\tilde{r}^X,\tilde{c}^Y)$, $k=1,\cdots,27$.

\begin{lemma} \label{lem5-2}
For $A\in \mathbb{C}^{n\times n}$ and
$X,Y \in D^P_n$, $\alpha,\beta\in[0,1]$,
let ${\hat\Gamma}^{(k)}(A,\star)$ and ${\bf\hat\Gamma}^{(k)}(A,\tilde{r},\tilde{c})$ be defined by Definition 5.2, $k=1,\cdots,31$. Then

\begin{itemize}
\item[\rm(1)]
${\hat\Gamma}^{(1)}(A,\tilde{r}^X) = {\hat\Gamma}^{(\xi)}(A,\tilde{r}^X,\tilde{r}^X,\alpha) = {\hat\Gamma}^{(\varsigma)}(A,\tilde{r}^X,1) = {\hat\Gamma}^{(\mu)}(A,\tilde{r}^X,\tilde{c}^Y,1) \\
= {\hat\Gamma}^{(\nu)}(A,\tilde{r}^X,\tilde{c}^Y,1,1)  = {\hat\Gamma}^{(\zeta)}(A,\tilde{r}^X,\tilde{c}^Y,0,1) = {\hat\Gamma}^{(\eta)}(A,\tilde{r}^X,\tilde{c}^Y,0,0) \\
\supseteq [{\hat\Gamma}^{(5)}(A,\tilde{r}^X,\tilde{c}^Y,\alpha)\cup{\hat\Gamma}^{(16)}(A,\tilde{r}^X,\tilde{c}^Y,\alpha)]$,

$\xi=3,4$, $\varsigma=16,19$, $\mu=3,4,18,21$, $\nu=22,\cdots,31$, \\ $\zeta=22,24,26,28,30$, $\eta=23,25,27,29,31$;

\item[\rm(2)]
${\hat\Gamma}^{(2)}(A,\tilde{c}^Y) = {\hat\Gamma}^{(\xi)}(A,\tilde{c}^Y,\tilde{c}^Y,\alpha) = {\hat\Gamma}^{(\varsigma)}(A,\tilde{c}^Y,0) = {\hat\Gamma}^{(\mu)}(A,\tilde{r}^X,\tilde{c}^Y,0) \\
= {\hat\Gamma}^{(\nu)}(A,\tilde{r}^X,\tilde{c}^Y,1,0) = {\hat\Gamma}^{(\zeta)}(A,\tilde{r}^X,\tilde{c}^Y,0,0)  = {\hat\Gamma}^{(\eta)}(A,\tilde{r}^X,\tilde{c}^Y,0,1) \\
\supseteq [{\hat\Gamma}^{(6)}(A,\tilde{r}^X,\tilde{c}^Y,\alpha)\cup{\hat\Gamma}^{(17)}(A,\tilde{r}^X,\tilde{c}^Y,\alpha)]$,

$\xi=3,4$, $\varsigma=17,20$, $\mu=3,4,18,21$, $\nu=22,\cdots,31$, \\ $\zeta=22,24,26,28,30$, $\eta=23,25,27,29,31$;

\item[\rm(3)]
${\hat\Gamma}^{(3)}(A,\tilde{r}^X,\tilde{c}^Y,\alpha) = {\hat\Gamma}^{(\mu)}(A,\tilde{r}^X,\tilde{c}^Y,1,\alpha) \\
\supseteq [{\hat\Gamma}^{(8)}(A,\tilde{r}^X,\tilde{c}^Y,\alpha)\cup{\hat\Gamma}^{(22)}(A,\tilde{r}^X,\tilde{c}^Y,\alpha,\alpha)]$,
$\mu=22,23,26,27,30,31$;

\item[\rm(4)]
${\hat\Gamma}^{(4)}(A,\tilde{r}^X,\tilde{c}^Y,\alpha) = {\hat\Gamma}^{(\mu)}(A,\tilde{r}^X,\tilde{c}^Y,1,\alpha) \\
\supseteq
[{\hat\Gamma}^{(\nu)}(A,\tilde{r}^X,\tilde{c}^Y,\alpha)\cup{\hat\Gamma}^{(28)}(A,\tilde{r}^X,\tilde{c}^Y,\alpha,\alpha)]$,
$\mu=24,25,28,29$, $\nu=3,12$;

\item[\rm(5)]
${\hat\Gamma}^{(5)}(A,\tilde{r}^X) = {\hat\Gamma}^{(7)}(A,\tilde{r}^X,\tilde{r}^X) = {\hat\Gamma}^{(16)}(A,\tilde{r}^X,\frac{1}{2}) = {\hat\Gamma}^{(\xi)}(A,\tilde{r}^X,\tilde{r}^X,\alpha) \\
= {\hat\Gamma}^{(\mu)}(A,\tilde{r}^X,\tilde{c}^Y,1) = {\hat\Gamma}^{(\nu)}(A,\tilde{r}^X,\tilde{c}^Y,\frac{1}{2},1)$,

$\xi=8,\cdots,13$, $\mu=8,10,12,14$, $\nu=22,24,28$;

\item[\rm(6)]
${\hat\Gamma}^{(6)}(A,\tilde{c}^Y) = {\hat\Gamma}^{(7)}(A,\tilde{c}^Y,\tilde{c}^Y) = {\hat\Gamma}^{(17)}(A,\tilde{c}^Y,\frac{1}{2}) = {\hat\Gamma}^{(\xi)}(A,\tilde{c}^Y,\tilde{c}^Y,\alpha) \\
= {\hat\Gamma}^{(\mu)}(A,\tilde{r}^X,\tilde{c}^Y,0) = {\hat\Gamma}^{(\nu)}(A,\tilde{r}^X,\tilde{c}^Y,\frac{1}{2},0)$,

$\xi=8,\cdots,13$, $\mu=8,10,12,14$, $\nu=22,24,28$;

\item[\rm(7)]
${\hat\Gamma}^{(7)}(A,\tilde{r}^X,\tilde{c}^Y) = {\hat\Gamma}^{(18)}(A,\tilde{r}^X,\tilde{c}^Y,\frac{1}{2}) = {\hat\Gamma}^{(\mu)}(A,\tilde{r}^X,\tilde{c}^Y,1) \\
= {\hat\Gamma}^{(\nu)}(A,\tilde{r}^X,\tilde{c}^Y,\frac{1}{2},1)$,
$\mu=9,11,13,15$, $\nu=23,25,29$;

\item[\rm(8)]
${\hat\Gamma}^{(8)}(A,\tilde{r}^X,\tilde{c}^Y,\alpha) = {\hat\Gamma}^{(22)}(A,\tilde{r}^X,\tilde{c}^Y,\frac{1}{2},\alpha) \subseteq {\hat\Gamma}^{(k)}(A,\tilde{r}^X,\tilde{c}^Y,\alpha)$,

$k=10,12,14$;

\item[\rm(9)]
${\hat\Gamma}^{(9)}(A,\tilde{r}^X,\tilde{c}^Y,\alpha) = {\hat\Gamma}^{(23)}(A,\tilde{r}^X,\tilde{c}^Y,\frac{1}{2},\alpha) \subseteq {\hat\Gamma}^{(k)}(A,\tilde{r}^X,\tilde{c}^Y,\alpha)$,

$k=11,13,15$;

\item[\rm(10)]
${\hat\Gamma}^{(12)}(A,\tilde{r}^X,\tilde{c}^Y,\alpha) = {\hat\Gamma}^{(28)}(A,\tilde{r}^X,\tilde{c}^Y,\frac{1}{2},\alpha)$;

\item[\rm(11)]
${\hat\Gamma}^{(13)}(A,\tilde{r}^X,\tilde{c}^Y,\alpha) = {\hat\Gamma}^{(29)}(A,\tilde{r}^X,\tilde{c}^Y,\frac{1}{2},\alpha)$;

\item[\rm(12)]
${\hat\Gamma}^{(16)}(A,\tilde{r}^X,\alpha) = {\hat\Gamma}^{(18)}(A,\tilde{r}^X,\tilde{r}^X,\alpha)  = {\hat\Gamma}^{(\mu)}(A,\tilde{r}^X,\tilde{r}^X,\alpha,\beta) \\
= {\hat\Gamma}^{(\nu)}(A,\tilde{r}^X,\tilde{c}^Y,\alpha,1) \subseteq {\hat\Gamma}^{(19)}(A,\tilde{r}^X,\tilde{c}^Y,\alpha) = {\hat\Gamma}^{(\zeta)}(A,\tilde{r}^X,\tilde{c}^Y,\alpha,1)$,

$\mu=22,23,24,25,28,29$, $\nu=22,24,28$, $\zeta=26,30$;

\item[\rm(13)]
${\hat\Gamma}^{(17)}(A,\tilde{c}^Y,\alpha) = {\hat\Gamma}^{(18)}(A,\tilde{c}^Y,\tilde{c}^Y,\alpha)  = {\hat\Gamma}^{(\mu)}(A,\tilde{c}^Y,\tilde{c}^Y,\alpha,\beta)\\
= {\hat\Gamma}^{(\nu)}(A,\tilde{r}^X,\tilde{c}^Y,\alpha,0) \subseteq {\hat\Gamma}^{(20)}(A,\tilde{c}^Y,\alpha) =  {\hat\Gamma}^{(\zeta)}(A,\tilde{r}^X,\tilde{c}^Y,\alpha,0)$,

$\mu=22,23,24,25,28,29$, $\nu=22,24,28$, $\zeta=26,30$;

\item[\rm(14)]
${\hat\Gamma}^{(18)}(A,\tilde{r}^X,\tilde{c}^Y,\alpha) = {\hat\Gamma}^{(\mu)}(A,\tilde{r}^X,\tilde{c}^Y,\alpha,1) \subseteq {\hat\Gamma}^{(21)}(A,\tilde{r}^X,\tilde{c}^Y,\alpha) \\
= {\hat\Gamma}^{(\nu)}(A,\tilde{r}^X,\tilde{c}^Y,\alpha,1)$,
$\mu=23,25,29$, $\nu=27,31$;

\item[\rm(15)]
${\hat\Gamma}^{(22)}(A,\tilde{r}^X,\tilde{c}^Y,\alpha,\beta) \subseteq {\hat\Gamma}^{(k)}(A,\tilde{r}^X,\tilde{c}^Y,\alpha,\beta)$,
$k=24,26,28,30$;

\item[\rm(16)]
${\hat\Gamma}^{(23)}(A,\tilde{r}^X,\tilde{c}^Y,\alpha,\beta) \subseteq {\hat\Gamma}^{(k)}(A,\tilde{r}^X,\tilde{c}^Y,\alpha,\beta)$,
$k=25,27,29,31$.
\end{itemize}

Furthermore, it holds that

\begin{itemize}
\item[\rm(a)]
${\bf\hat\Gamma}^{(1)}(A,\tilde{r},\tilde{c}) = {\bf\hat\Gamma}^{(k)}(A,\tilde{r},\tilde{c})$,
$k=2,\cdots,6,8,10,12,14$;

\item[\rm(b)]
${\bf\hat\Gamma}^{(7)}(A,\tilde{r},\tilde{c}) = {\bf\hat\Gamma}^{(k)}(A,\tilde{r},\tilde{c})$,
$k=9,11,13,15$;

\item[\rm(c)]
${\bf\hat\Gamma}^{(16)}(A,\tilde{r},\tilde{c})={\bf\hat\Gamma}^{(k)}(A,\tilde{r},\tilde{c})$,
$k=17,22,24,28$;

\item[\rm(d)]
${\bf\hat\Gamma}^{(18)}(A,\tilde{r},\tilde{c}) ={\bf\hat\Gamma}^{(k)}(A,\tilde{r},\tilde{c})$,
$k=23,25,29$;

\item[\rm(e)]
${\bf\hat\Gamma}^{(19)}(A,) = {\bf\hat\Gamma}^{(20)}(A,\tilde{r},\tilde{c}) = {\bf\hat\Gamma}^{(26)}(A,\tilde{r},\tilde{c})$;

\item[\rm(f)]
${\bf\hat\Gamma}^{(21)}(A,\tilde{r},\tilde{c}) = {\bf\hat\Gamma}^{(27)}(A,\tilde{r},\tilde{c})$;

\item[\rm(g)]
${\bf\hat\Gamma}^{(18)}(A,\tilde{r},\tilde{c}) \subseteq {\bf\hat\Gamma}^{(7)}(A,\tilde{r},\tilde{c}) \subseteq {\bf\hat\Gamma}^{(1)}(A,\tilde{r},\tilde{c})$;

\item[\rm(h)]
${\bf\hat\Gamma}^{(18)}(A,\tilde{r},\tilde{c}) \subseteq {\bf\hat\Gamma}^{(16)}(A,\tilde{r},\tilde{c}) \subseteq {\bf\hat\Gamma}^{(30)}(A,\tilde{r},\tilde{c}) \subseteq {\bf\hat\Gamma}^{(19)}(A,\tilde{r},\tilde{c}) \\ \subseteq {\bf\hat\Gamma}^{(1)}(A,\tilde{r},\tilde{c})$;

\item[\rm(i)]
${\bf\hat\Gamma}^{(18)}(A,\tilde{r},\tilde{c}) \subseteq {\bf\hat\Gamma}^{(31)}(A,\tilde{r},\tilde{c}) \subseteq {\bf\hat\Gamma}^{(21)}(A,\tilde{r},\tilde{c}) \subseteq {\bf\hat\Gamma}^{(19)}(A,\tilde{r},\tilde{c})$.
\end{itemize}
\end{lemma}

\begin{proof}
Set $(g,h)=(\tilde{r}^X,\tilde{c}^Y)$. Then from (1)-(13) of Lemma 5.1 we can derive (1)-(16) immediately.

In order to prove (a)-(i), by Lemma 5.1, we just need to prove that
${\bf\hat\Gamma}^{(k)}(A,\tilde{r},\\ \tilde{c})
= {\bf\hat\Gamma}^{(k+1)}(A,\tilde{r},\tilde{c})$, $k=1,5,16,19$.

In fact, for any $Y \in D^P_n$, by Lemmas 2.2 and 2.3, there exists $X(Y) \in D^P_n$ such that
$\tilde{c}^Y(A,\tilde{r},\tilde{c}) \ge \tilde{r}^{X(Y)}(A,\tilde{r},\tilde{c})$, so that
$\rho_i^{(2)}(A,\tilde{c}^Y) \ge \rho_i^{(2)}(A,\tilde{r}^{X(Y)}) = \rho_i^{(1)}(A,\tilde{r}^{X(Y)})$ for all $i\in \cal N$,
which implies that ${\hat\Gamma}^{(2)}(A,\tilde{c}^Y)\supseteq{\hat\Gamma}^{(1)}(A,\tilde{r}^{X(Y)})$ and therefore ${\bf\hat\Gamma}^{(2)}(A,\tilde{r},\tilde{c})\supseteq{\bf\hat\Gamma}^{(1)}(A,\tilde{r},\tilde{c})$.
Evidenced by the same token,
${\bf\hat\Gamma}^{(1)}(A,\tilde{r},\tilde{c})\\ \supseteq{\bf\hat\Gamma}^{(2)}(A,\tilde{r},\tilde{c})$.

The proofs of ${\bf\hat\Gamma}^{(k)}(A,\tilde{r},\tilde{c})
= {\bf\hat\Gamma}^{(k+1)}(A,\tilde{r},\tilde{c})$, $k=5,16,19$, are completely similar.
\end{proof}

Similar to Theorem 5.3, by Lemma 5.2, the following theorem is obvious.

\begin{theorem} \label{thm5-4}
For $A\in \mathbb{C}^{n\times n}$ and
$X,Y \in D^P_n$, $\alpha,\beta\in[0,1]$,
let ${\hat\Gamma}^{(k)}(A,\star)$ and ${\bf\hat\Gamma}^{(k)}(A,\tilde{r},\tilde{c})$ be defined by Definition 5.2, $k=1,\cdots,31$. Then the following results hold:

\begin{itemize}

\item[\rm(a)]
For any $\lambda \in \sigma(A)$, there exists $i_k \in {\cal N}$ such that

$\lambda \in
\hat\Gamma_{i_k}^{(k)}(A,\star)$ for $k=1,2,3,4$,

and there exist $i_k, j_k \in {\cal N}$ with $i_k\not=j_k$ such that

$\lambda \in
\hat\Gamma_{i_k,j_k}^{(k)}(A,\star)$ for $k=5,\cdots,31$;

\item[\rm(b)]
$\sigma(A) \subseteq {\bf\hat\Gamma}^{(k)}(A,\tilde{r},\tilde{c}) \subseteq {\hat\Gamma}^{(k)}(A,\star)$ for
$k=1,\cdots,31$;

\item[\rm(c)]
$\sigma(A) \subseteq \bigcap\limits_{1\le k \le 31}{\bf\hat\Gamma}^{(k)}(A,\tilde{r},\tilde{c}) = {\bf\hat\Gamma}^{(18)}(A,\tilde{r},\tilde{c})$.
\end{itemize}
\end{theorem}

It is easy to prove that ${\bf\hat\Gamma}^{(18)}(A,\tilde{r},\tilde{c})={\bf\check\Gamma}^{(15)}(A,\mathcal{G}_n)$.
Hence by the theorem above and Theorem 5.3, we see that ${\bf\hat\Gamma}^{(18)}(A,\tilde{r},\tilde{c})$
is the minimal eigenvalue inclusion set among ${\bf\hat\Gamma}^{(k)}(A,\tilde{r},\tilde{c})$,
${\hat\Gamma}^{(k)}(A,\star)$, $k=1,\cdots,31$, and ${\bf\check\Gamma}^{(k)}(A,\mathcal{G}_n)$,
${\check\Gamma}^{(k)}(A,\ast)$, $k=1,\cdots,27$, where ${\bf\hat\Gamma}^{(18)}(A,\tilde{r},\tilde{c}) ={\bf\hat\Gamma}^{(\mu)}(A,\tilde{r},\tilde{c})={\bf\check\Gamma}^{(\nu)}(A,\mathcal{G}_n)$, $\mu=23,25,29$, $\nu=15,19,21,25$.

In Definition 5.2 if we take $r$ and $c$ instant of $\tilde{r}$ and $\tilde{c}$, respectively, then we can
definition another 31 kind of circles and ovals of Cassini and their sets.

\begin{definition} \label{den5-3}
For $A\in \mathbb{C}^{n\times n}$ and
$X,Y \in D^P_n$, $\alpha,\beta\in[0,1]$, $i,j \in {\cal N}$, $i \not= j$,
just as Definition 5.2, we define 31 kind of circles, ovals of Cassini and their sets ${\hat\Gamma}^{(k)}(A,\star)$ and ${\bf\hat\Gamma}^{(k)}(A,r,c)$, $k=1,\cdots,31$, where the symbol ``$\star$"
denotes ``$r^X$", ``$c^Y$", ``$r^X,\alpha$", ``$c^Y,\alpha$", ``$r^X,c^Y$", ``$r^X,c^Y,\alpha$", ``$r^X,c^Y,\alpha,\beta$", respectively.
\end{definition}

When $k=1$, $\hat{\Gamma}_i^{(1)}(A,r^X)$, ${\hat\Gamma}^{(1)}(A,r^X)$
and ${\bf\hat\Gamma}^{(1)}(A,r,c)$
are proposed by Varga [48], where ${\bf\hat\Gamma}^{(1)}(A,r,c)$ is called the minimal Ger\v{s}gorin set of $A$.
In [53], a numerical approximation to the minimal Ger\v{s}gorin set of an irreducible matrix is given.

The following lemma is easy to prove.

\begin{lemma} \label{lem5-3}
For $A\in \mathbb{C}^{n\times n}$ and $\alpha,\beta\in[0,1]$, let $\hat{\Gamma}^{(k)}(A,\star)$
and ${\bf\hat\Gamma}^{(k)}(A,r,c)$ be defined by Definition 5.3, $k=1,\cdots,31$. Then

\begin{itemize}
\item[\rm(a)] The relationships (1)-(16) in Lemma 5.2 are valid, where $\tilde{r}$ and $\tilde{c}$ are changed into $r$ and $c$, respectively;

\item[\rm(b)] ${\bf\hat\Gamma}^{(22)}(A,r,c)\subseteq{\bf\hat\Gamma}^{(k)}(A,r,c)$,

$k=1,\cdots,6,8,10,12,14,16,17,19,20,24,26,28,30$;

\item[\rm(c)] ${\bf\hat\Gamma}^{(23)}(A,r,c)\subseteq{\bf\hat\Gamma}^{(k)}(A,r,c)$,

$k=1,2,3,4,7,9,11,13,15,18,21,25,27,29,31$.
  \end{itemize}
\end{lemma}

Similar to Theorem 5.4, by Lemma 5.3, we have the following theorem.

\begin{theorem} \label{thm5-5}
For $A\in \mathbb{C}^{n\times n}$ and
$X,Y \in D^P_n$, $\alpha,\beta\in[0,1]$,
let ${\hat\Gamma}^{(k)}(A,\star)$ and ${\bf\hat\Gamma}^{(k)}(A,r,c)$
be defined by Definition 5.3. Then the following results hold:

\begin{itemize}

\item[\rm(a)]
For any $\lambda \in \sigma(A)$, there exists $i_k \in {\cal N}$ such that

$\lambda \in
\hat{\Gamma}_{i_k}^{(k)}(A,\star)$ for $k=1,2,3,4$,

and there exist $i_k, j_k \in {\cal N}$ with $i_k\not=j_k$ such that

$\lambda \in
\hat{\Gamma}_{i_k,j_k}^{(k)}(A,\star)$ for $k=5,\cdots,31$;

\item[\rm(b)]
$\sigma(A) \subseteq {\bf\hat\Gamma}^{(k)}(A,r,c) \subseteq {\hat\Gamma}^{(k)}(A,\star)$ for $k=1,\cdots,31$;

\item[\rm(c)]
$\sigma(A) \subseteq \bigcap\limits_{1\le k \le 31}{\bf\hat\Gamma}^{(k)}(A,r,c) = [{\bf\hat\Gamma}^{(22)}(A,r,c)\cap{\bf\hat\Gamma}^{(23)}(A,r,c)]$.
\end{itemize}
\end{theorem}

This theorem tells us that ${\bf\hat\Gamma}^{(22)}(A,r,c)\cap{\bf\hat\Gamma}^{(23)}(A,r,c)$ is the minimal eigenvalue
inclusion set compared with ${\bf\hat\Gamma}^{(k)}(A,r,c)$ and ${\hat\Gamma}^{(k)}(A,\star)$,
$k=1,\cdots,31$, so that it is always at least as good as the
minimal Ger\v{s}gorin set defined by [48].

When $X=Y=I$ we define the circles, ovals of Cassini and their sets as follows.

\begin{definition} \label{den5-4}
For $A\in \mathbb{C}^{n\times n}$ and
$X,Y \in D^P_n$, $\alpha,\beta\in[0,1]$, $i,j \in {\cal N}$, $i \not= j$,
we define 31 kind of circles, ovals of Cassini and their sets as follows, where the symbol ``$\star$"
denotes ``$\tilde{r}$", ``$\tilde{c}$", ``$\tilde{r},\alpha$", ``$\tilde{c},\alpha$", ``$\tilde{r},\tilde{c}$", ``$\tilde{r},\tilde{c},\alpha$", ``$\tilde{r},\tilde{c},\alpha,\beta$", respectively.

\begin{itemize}
\item[$\bullet$] $\Gamma_i^{(k)}(A,\star):=
\{z\in \mathbb{C}^n: |z - a_{i,i}| \le \rho_i^{(k)}(A,\star)\}$, $k=1,2,3,4$, with

$\rho_i^{(1)}(A,\tilde{r}) = \tilde{r}_i(A)$,

$\rho_i^{(2)}(A,\tilde{c}) = \tilde{c}_i(A)$,

$\rho_i^{(3)}(A,\tilde{r},\tilde{c},\alpha) = [\tilde{r}_i(A)]^{\alpha}[\tilde{c}_i(A)]^{1-\alpha}$,

$\rho_i^{(4)}(A,\tilde{r},\tilde{c},\alpha) = \alpha\tilde{r}_i(A) + (1-\alpha)\tilde{c}_i(A)$;

\item[$\circ$] $\Gamma_{i,j}^{(k)}(A,\star):=
\{z\in \mathbb{C}^n: |z - a_{i,i}| |z - a_{j,j}| \le \rho_{i,j}^{(k)}(A,\star)\}$,
$k=5,\cdots,15$, with

$\rho_{i,j}^{(5)}(A,\tilde{r}) =
\tilde{r}_i(A)\tilde{r}_j(A)$,

$\rho_{i,j}^{(6)}(A,\tilde{c}) =
\tilde{c}_i(A)\tilde{c}_j(A)$,

$\rho_{i,j}^{(7)}(A,\tilde{r},\tilde{c}) =
\tilde{r}_i(A)\tilde{c}_j(A)$,

$\rho_{i,j}^{(8)}(A,\tilde{r},\tilde{c},\alpha) =
[\tilde{r}_i(A)\tilde{r}_j(A)]^{\alpha}[\tilde{c}_i(A)\tilde{c}_j(A)]^{1-\alpha}$,

$\rho_{i,j}^{(9)}(A,\tilde{r},\tilde{c},\alpha) =
[\tilde{r}_i(A)\tilde{c}_j(A)]^{\alpha}[\tilde{r}_j(A)\tilde{c}_i(A)]^{1-\alpha}$,

$\rho_{i,j}^{(10)}(A,\tilde{r},\tilde{c},\alpha) =
\alpha\tilde{r}_i(A)\tilde{r}_j(A) + (1-\alpha)\tilde{c}_i(A)\tilde{c}_j(A)$,

$\rho_{i,j}^{(11)}(A,\tilde{r},\tilde{c},\alpha) =
\alpha\tilde{r}_i(A)\tilde{c}_j(A) + (1-\alpha)\tilde{r}_j(A)\tilde{c}_i(A)$,

$\rho_{i,j}^{(12)}(A,\tilde{r},\tilde{c},\alpha) =
[\alpha\tilde{r}_i(A)+(1-\alpha)\tilde{c}_i(A)][\alpha\tilde{r}_j(A) + (1-\alpha)\tilde{c}_j(A)]$,

$\rho_{i,j}^{(13)}(A,\tilde{r},\tilde{c},\alpha) =
[\alpha\tilde{r}_i(A)+(1-\alpha)\tilde{c}_i(A)][\alpha\tilde{c}_j(A) + (1-\alpha)\tilde{r}_j(A)]$,

$\rho_{i,j}^{(14)}(A,\tilde{r},\tilde{c},\alpha) =
[\alpha\tilde{r}_i(A)+(1-\alpha)\tilde{c}_j(A)][\alpha\tilde{r}_j(A) + (1-\alpha)\tilde{c}_i(A)]$,

$\rho_{i,j}^{(15)}(A,\tilde{r},\tilde{c},\alpha) =
[\alpha\tilde{r}_i(A)+(1-\alpha)\tilde{r}_j(A)][\alpha\tilde{c}_j(A) + (1-\alpha)\tilde{c}_i(A)]$;

\item[$\circ$] $\Gamma_{i,j}^{(k)}(A,\star):=
\{z\in \mathbb{C}^n: |z - a_{i,i}|^{\alpha}|z - a_{j,j}|^{1-\alpha} \le \rho_{i,j}^{(k)}(A,\star)\}$,
$k=16,\cdots,31$, with

$\rho_{i,j}^{(16)}(A,\tilde{r},\alpha) =
[\tilde{r}_i(A)]^{\alpha}[\tilde{r}_j(A)]^{1-\alpha}$,

$\rho_{i,j}^{(17)}(A,\tilde{c},\alpha) =
[\tilde{c}_i(A)]^{\alpha}[\tilde{c}_j(A)]^{1-\alpha}$;

$\rho_{i,j}^{(18)}(A,\tilde{r},\tilde{c},\alpha) =
[\tilde{r}_i(A)]^{\alpha}[\tilde{c}_j(A)]^{1-\alpha}$,

$\rho_{i,j}^{(19)}(A,\tilde{r},\alpha) =
\alpha \tilde{r}_i(A)+(1-\alpha)\tilde{r}_j(A)$,

$\rho_{i,j}^{(20)}(A,\tilde{c},\alpha) =
\alpha \tilde{c}_i(A)+(1-\alpha)\tilde{c}_j(A)$,

$\rho_{i,j}^{(21)}(A,\tilde{r},\tilde{c},\alpha) =
\alpha \tilde{r}_i(A)+(1-\alpha)\tilde{c}_j(A)$,

$\rho_{i,j}^{(22)}(A,\tilde{r},\tilde{c},\alpha,\beta) =
([\tilde{r}_i(A)]^{\beta}[\tilde{c}_i(A)]^{1-\beta})^\alpha ([\tilde{r}_j(A)]^{\beta}[\tilde{c}_j(A)]^{1-\beta})^{1-\alpha}$,

$\rho_{i,j}^{(23)}(A,\tilde{r},\tilde{c},\alpha,\beta) =
([\tilde{r}_i(A)]^{\beta}[\tilde{c}_i(A)]^{1-\beta})^\alpha
([\tilde{c}_j(A)]^{\beta} [\tilde{r}_j(A)]^{1-\beta})^{1-\alpha}$,

$\rho_{i,j}^{(24)}(A,\tilde{r},\tilde{c},\alpha,\beta) =
\beta [\tilde{r}_i(A)]^\alpha[\tilde{r}_j(A)]^{1-\alpha} + (1-\beta)[\tilde{c}_i(A)]^\alpha [\tilde{c}_j(A)]^{1-\alpha}$,

$\rho_{i,j}^{(25)}(A,\tilde{r},\tilde{c},\alpha,\beta) =
\beta [\tilde{r}_i(A)]^\alpha [\tilde{c}_j(A)]^{1-\alpha} + (1-\beta)[\tilde{c}_i(A)]^\alpha [\tilde{r}_j(A)]^{1-\alpha}$,

$\rho_{i,j}^{(26)}(A,\tilde{r},\tilde{c},\alpha,\beta) =
\alpha [\tilde{r}_i(A)]^{\beta}[\tilde{c}_i(A)]^{1-\beta} + (1-\alpha)[\tilde{r}_j(A)]^{\beta}[\tilde{c}_j(A)]^{1-\beta}$,

$\rho_{i,j}^{(27)}(A,\tilde{r},\tilde{c},\alpha,\beta) =
\alpha [\tilde{r}_i(A)]^{\beta}[\tilde{c}_i(A)]^{1-\beta} + (1-\alpha)[\tilde{c}_j(A)]^{\beta}[\tilde{r}_j(A)]^{1-\beta}$,

$\rho_{i,j}^{(28)}(A,\tilde{r},\tilde{c},\alpha,\beta) =
[\beta \tilde{r}_i(A)+(1-\beta)\tilde{c}_i(A)]^\alpha[\beta \tilde{r}_j(A) + (1-\beta)\tilde{c}_j(A)]^{1-\alpha}$,

$\rho_{i,j}^{(29)}(A,\tilde{r},\tilde{c},\alpha,\beta) =
[\beta \tilde{r}_i(A)+(1-\beta)\tilde{c}_i(A)]^\alpha[\beta \tilde{c}_j(A) + (1-\beta)\tilde{r}_j(A)]^{1-\alpha}$,

$\rho_{i,j}^{(30)}(A,\tilde{r},\tilde{c},\alpha,\beta) =
[\alpha \tilde{r}_i(A)+(1-\alpha)\tilde{r}_j(A)]^{\beta}[\alpha \tilde{c}_i(A) + (1-\alpha)\tilde{c}_j(A)]^{1-\beta}$,

$\rho_{i,j}^{(31)}(A,\tilde{r},\tilde{c},\alpha,\beta) =
[\alpha \tilde{r}_i(A)+(1-\alpha)\tilde{c}_j(A)]^{\beta}[\alpha \tilde{c}_i(A) + (1-\alpha)\tilde{r}_j(A)]^{1-\beta}$
\end{itemize}
and
  \begin{itemize}

 \item[$\bullet$] $\Gamma^{(k)}(A,\star):= \bigcup\limits_{i=1}^n\Gamma_i^{(k)}(A,\star)$, $k=1,2,3,4$,

$\Gamma^{(k)}(A,\star):= \bigcup\limits_{i,j=1,i\not=j}^n\Gamma_{i,j}^{(k)}(A,\star)$,
$k=5, \cdots, 31$;

\item[$\bullet$]
${\bf\Gamma}^{(k)}(A,\tilde{r},\tilde{c}):= \Gamma^{(k)}(A,\star)$, $k=1,2,5,6,7$,

${\bf\Gamma}^{(k)}(A,\tilde{r},\tilde{c}):= \bigcap\limits_{\alpha\in[0,1]}\Gamma^{(k)}(A,\star)$,
$k=3,4,8,\cdots,21$,

${\bf\Gamma}^{(k)}(A,\tilde{r},\tilde{c}):=
\bigcap\limits_{\alpha,\beta\in[0,1]}\Gamma^{(k)}(A,\tilde{r},\tilde{c},\alpha,\beta)$,
$k=22,\cdots,31$.
 \end{itemize}
\end{definition}

Among ${\Gamma}^{(k)}(A,\star)$ and ${\bf\Gamma}^{(k)}(A,r,c)$ there are the same relationships with Lemma 5.3, where $\hat\Gamma$, $\bf\hat\Gamma$, $r$ and $c$ are changed into $\Gamma$, $\bf\Gamma$, $\tilde{r}$ and $\tilde{c}$, respectively.

The following theorem is easy to prove.

\begin{theorem} \label{thm5-6}
For $A\in \mathbb{C}^{n\times n}$ and $\alpha,\beta\in[0,1]$, let $\Gamma^{(k)}(A,\star)$
and ${\bf\Gamma}^{(k)}(A,\tilde{r}, \tilde{c})$ be defined by Definition 5.4, $k=1,\cdots,31$.
Then the following results hold:

\begin{itemize}
\item[\rm(a)]
For any $\lambda \in \sigma(A)$, there exists $i_k \in {\cal N}$ such that

$\lambda \in
\Gamma_{i_k}^{(k)}(A,\star)$ for $k=1,2,3,4$,

and there exist $i_k, j_k \in {\cal N}$ with $i_k\not=j_k$ such that

$\lambda \in
\Gamma_{i_k,j_k}^{(k)}(A,\star)$ for $k=5,\cdots,31$;

\item[\rm(b)]
$\sigma(A) \subseteq {\bf\Gamma}^{(k)}(A,\tilde{r},\tilde{c}) \subseteq \Gamma^{(k)}(A,\star)$ for
$k=1,\cdots,31$;

\item[\rm(c)]
$\sigma(A) \subseteq \bigcap\limits_{1\le k \le 31}{\bf\Gamma}^{(k)}(A,\tilde{r},\tilde{c}) =
[{\bf\Gamma}^{(22)}(A,\tilde{r},\tilde{c})\cap{\bf\Gamma}^{(23)}(A,\tilde{r},\tilde{c})]$.
\end{itemize}
\end{theorem}

This theorem tells us that ${\bf\Gamma}^{(22)}(A,\tilde{r},\tilde{c})\cap{\bf\Gamma}^{(23)}(A,\tilde{r},\tilde{c})$ is the minimal
eigenvalue inclusion set compared with ${\bf\Gamma}^{(k)}(A,\tilde{r},\tilde{c})$ and
$\Gamma^{(k)}(A,\star)$, $k=1,\cdots,31$.

At last, when $\tilde{r}$ and $\tilde{c}$ are replaced by $r$ and $c$ respectively,
we can define $\Gamma^{(k)}(A,\star)$
and ${\bf\Gamma}^{(k)}(A,r,c)$ by Definition 5.4.

\begin{definition} \label{den5-5}
For $A\in \mathbb{C}^{n\times n}$ and
$\alpha,\beta\in[0,1]$, $i,j \in {\cal N}$, $i \not= j$,
just as Definition 5.4, we define 31 kind of circles, ovals of Cassini and their sets ${\Gamma}^{(k)}(A,\star)$ and ${\bf\Gamma}^{(k)}(A,r,c)$, $k=1,\cdots,31$, where the symbol ``$\star$"
denotes ``$r$", ``$c$", ``$r,\alpha$", ``$c,\alpha$", ``$r,c$", ``$r,c,\alpha$", ``$r,c,\alpha,\beta$", respectively.
\end{definition}

There are the corresponding
relationships among ${\Gamma}^{(k)}(A,\star)$ and ${\bf\Gamma}^{(k)}(A,r,c)$, $k=1,\cdots,31$. The following theorem is directly.

\begin{theorem} \label{thm5-7}
For $A\in \mathbb{C}^{n\times n}$ and $\alpha,\beta\in[0,1]$, let $\Gamma^{(k)}(A,\star)$
and ${\bf\Gamma}^{(k)}(A,r,c)$ be defined as Definition 5.5, $k=1,\cdots,31$.
Then the following results hold:

\begin{itemize}
\item[\rm(a)]
For any $\lambda \in \sigma(A)$, there exists $i_k \in {\cal N}$ such that

$\lambda \in
\Gamma_{i_k}^{(k)}(A,\star)$ for $k=1,2,3,4$,

and there exist $i_k, j_k \in {\cal N}$ with $i_k\not=j_k$ such that

$\lambda \in
\Gamma_{i_k,j_k}^{(k)}(A,\star)$ for $k=5,\cdots,31$;

\item[\rm(b)]
$\sigma(A) \subseteq {\bf\Gamma}^{(k)}(A,r,c) \subseteq \Gamma^{(k)}(A,\star)$ for
$k=1,\cdots,31$;

\item[\rm(c)]
$\sigma(A) \subseteq \bigcap\limits_{1\le k \le 31}{\bf\Gamma}^{(k)}(A,\tilde{r},\tilde{c}) =
[{\bf\Gamma}^{(22)}(A,r,c)\cap{\bf\Gamma}^{(23)}(A,r,c)]$.
\end{itemize}
\end{theorem}

This theorem tells us that ${\bf\Gamma}^{(22)}(A,r,c) \cap {\bf\Gamma}^{(23)}(A,r,c)$ is the minimal eigenvalue
inclusion set compared with ${\bf\Gamma}^{(k)}(A,r,c)$ and
$\Gamma^{(k)}(A,\star)$, $k=1,\cdots,31$.

This theorem includes some known results as follows.

When $k=1$, it is exactly the Ger\v{s}gorin circle theorem.

It is shown that for any $\lambda \in \sigma(A)$, there exist $i_k \in {\cal N}$, $k=1,2,3$,
such that

\begin{itemize}
\item[]
$\lambda \in \Gamma_{i_1}^{(1)}(A,r)$ and $\lambda \in \Gamma_{i_2}^{(2)}(A,c)$ [4, Theorem 1];

\item[]
$\lambda \in \Gamma_{i_3}^{(3)}(A,r,c,\alpha)$ [35, Satz III],
\end{itemize}

there exist $i_k, j_k \in {\cal N}$ with $i_k\not=j_k$, $k=5,6,8$, such that

\begin{itemize}
\item[]
$\lambda \in \Gamma_{i_5,j_5}^{(5)}(A,r)$ and $\lambda \in \Gamma_{i_6,j_6}^{(6)}(A,c)$ [5, Theorem 11];

\item[]
$\lambda \in \Gamma_{i_8,j_8}^{(8)}(A,r,c,\alpha)$ [35, Satz V]
\end{itemize}

and
$\sigma(A)\subseteq {\bf\Gamma}^{(k)}(A,r,c)$, $k=3,4$ [11, Theorems 14,15].

Corresponding to $k=5,6$, when $k=7$ a new simple oval of Cassini is proposed in Theorem 5.7.
In general, there is no relationship between the three.

\newpage\markboth{}{}
\section{Conclusion}\label{se6}

In this paper, in order to investigate the strictly diagonally dominant matrices and
the inclusion regions of matrix eigenvalues, a class of $G$-function pairs is
proposed, which extends the concept of $G$-functions. Just as the definition of $G$-functions, we first
definite a positive monotonic function, then the corresponding $G$-function pair is induced.
For general $G$-function pairs, we prove their relations with strictly diagonally dominant
matrices and the inclusion regions of matrix eigenvalues, respectively. Thirteen kind of
$G$-function pairs are established, which are easy to determine numerically.
Their properties and characteristics are studied, and their relations with
$G$-functions are discussed. By using these special $G$-function pairs, we construct
a large of necessary and sufficient conditions for strictly diagonally dominant matrices
and the inclusion regions of matrix eigenvalues. These conditions and inclusion regions
are dependent on only $G$-function pairs and
$r$, $c$, $r^X$, $c^Y$, $\tilde{r}$, $\tilde{c}$, $\tilde{r}^X$, $\tilde{c}^Y$, for $X,Y \in D^P_n$,
so that they are easy to compute. Our results extend, include and are better than some classical
results.

In Section 3, we establish 13 kind of $G$-function pairs. Of cause,
we can define more $G$-function pairs according Definitions 3.1 and 3.2.

In Sections 4 and 5, a large of necessary and sufficient conditions
for strictly diagonally dominant matrices and matrix eigenvalue inclusion regions are
constructed by some special $G$-function pairs those are
respectively composed of $g,h \in \{\mathcal{G}_n, r, c, \tilde{r},
\tilde{c}, r^X, c^Y, \tilde{r}^X, \tilde{c}^Y\}$, where $X,Y \in D^P_n$.
How to construct more necessary and sufficient conditions, including regions is still an interesting subject.

For example, let
\begin{eqnarray*}
\frac{1}{p} + \frac{1}{q} = 1, \; p \ge 1, q\ge 1.
 \end{eqnarray*}

Like [38.39], if for $0 \le \alpha \le 1$, we take
\begin{eqnarray*}
g^{(1)}(\alpha) = (g^{(1)}_1(\alpha), \cdots, g^{(1)}_n(\alpha))^T
 \end{eqnarray*}
with
\begin{eqnarray*}
g^{(1)}_k(\alpha) = r_{k,{\alpha p}}^\alpha c_{k,(1-\alpha )q}^{1-\alpha}, \; k\in{\cal N},
 \end{eqnarray*}
and for $s>0$,
\begin{eqnarray*}
r_{k,s}= \left(\sum\limits_{j\in {\cal N}\backslash
\{k\}}|a_{k,j}|^s\right)^{1/s},  \; c_{k,s}= \left(\sum\limits_{i\in {\cal N}\backslash
\{k\}}|a_{i,k}|^s\right)^{1/s}, \; k\in{\cal N}.
 \end{eqnarray*}
Then by [38, Theorem I], $g^{(1)}(\alpha) \in \mathcal{G}_n$.

Similarly, due to Ostrowski (1951) [10,13], if we take
\begin{eqnarray*}
g^{(2)}(\bar{\alpha}) = (g^{(2)}_1(\alpha_1), \cdots, g^{(2)}_n(\alpha_n))^T
 \end{eqnarray*}
with
\begin{eqnarray*}
g^{(2)}_k(\alpha_k) = \alpha_k^{1/q}r_{k,p}, \; k\in{\cal N},
 \end{eqnarray*}
where $\bar\alpha = (\alpha_1, \cdots, \alpha_n)^T$ satisfies $\alpha_k > 0$, $k\in {\cal N}$, and
$\sum_{k=1}^n 1/(1+\alpha_k) \le 1$.
Then by [10, Theorem A], [13] or [52, Theorom 1.19], $g^{(2)}(\bar{\alpha}) \in \mathcal{G}_n$.

In like wise, if we take
\begin{eqnarray*}
g^{(3)}(\bar{\alpha}) = (g^{(3)}_1(\alpha_1), \cdots, g^{(3)}_n(\alpha_n))^T
 \end{eqnarray*}
with
\begin{eqnarray*}
g^{(3)}_k(\alpha_k) = r_{k,p}/\alpha_k, \; k\in{\cal N},
 \end{eqnarray*}
where $\alpha_k > 0$, $k\in {\cal N}$, and
$\sum_{k=1}^n \alpha_k^q \le 1$.
Then by [29, Theorom 4], $g^{(3)}(\bar{\alpha}) \in \mathcal{G}_n$.

Again, we take
\begin{eqnarray*}
g^{(4)}(\alpha) = (g^{(4)}_1(\alpha), \cdots, g^{(4)}_n(\alpha))^T
 \end{eqnarray*}
with
\begin{eqnarray*}
g^{(4)}_k(\alpha) = \alpha \max\limits_{j\in {\cal N}\backslash\{i\}}|a_{i,j}|, \; k\in{\cal N},
 \end{eqnarray*}
where $\alpha > 0$ satisfies
$\sum_{k=1}^n r_i(A)/\max _{j\in {\cal N}\backslash\{i\}}|a_{i,j}| \le \alpha(1+\alpha)$.
Then by [14, Corollary 1.2] or [52, Theorem 1.20], $g^{(4)}(\alpha) \in \mathcal{G}_n$.

Therefore, using $g^{(k)}(\alpha)$, $g^{(2)}(\bar{\alpha})$, $g^{(3)}(\bar{\alpha})$ and $g^{(4)}(\alpha)$,
we can derive more sufficient conditions for
the generalized diagonally dominant matrices
and the matrix eigenvalue inclusion regions.

In order to generalize strict diagonal dominance to the partitioned matrices, it was almost
simultaneously and independently considered [16, 18, 39].
For the block strictly (or strictly block) diagonally dominant matrices defined
[16, 52], Just as block $G$-functions [8] or $G$-functions
in the partitioned case [52], we can propose block $G$-function pairs [47].
The corresponding necessary and sufficient conditions for the block strictly diagonally
dominant matrices and the matrix eigenvalue inclusion regions can be obtained.

In the paper, the $G$-function pairs is defined on ${\mathbb{C}^{n\times n}}$.
We can define it on some subspace of ${\mathbb{C}^{n\times n}}$ [47].
How to extend $G$-functions and define other function pairs similar to $G$-function
pairs to describe the irreducible diagonally dominant matrices and their eigenvalue inclusion regions
is an interesting subject.

Recently, the Ger\v{s}gorin circle theorem has been extended to tensors. It is also an
interesting open problem how to generalize the $G$-function pairs to tensors and derive
the corresponding eigenvalue inclusion regions.

\end{document}